\documentclass[12pt]{amsart}

\usepackage{fullpage}
\usepackage{mathtools}
\usepackage{amssymb,amsmath,amsthm,amscd,mathrsfs,graphicx}
\usepackage[dvipsnames]{xcolor}
\usepackage{bm}
\usepackage{dsfont}
\usepackage{enumerate}
\usepackage{epsfig}
\usepackage{float, graphicx}
\usepackage{latexsym, amsxtra}
\usepackage{pdfpages}
\usepackage{multicol}
\usepackage[normalem]{ulem}
\usepackage{psfrag}
\usepackage{stmaryrd}
\usepackage{tikz}
\usepackage[T1]{fontenc}
\usepackage{url}
\usepackage{verbatim}
\usepackage{xy}
\usepackage{indentfirst}
\usepackage{tikz-cd}
\tikzset{%
    symbol/.style={%
        ,draw=none
        ,every to/.append style={%
            edge node={node [sloped, allow upside down, auto=false]{$#1$}}}
    }
}
\usepackage{url}
\usepackage{longtable}

\usepackage[
  colorlinks=true,
  linkcolor=red,
  urlcolor=blue,
  citecolor=blue
]{hyperref}

\makeatletter
\def\thm@space@setup{%
  \thm@preskip=2ex \thm@postskip=2ex
}
\makeatother

\makeatletter

\newcommand{\Rmnum}[1]{\expandafter\@slowromancap\romannumeral #1@}
\makeatother

\newcommand{\Id}{{\mathrm{Id}}}
\numberwithin{equation}{section}
\theoremstyle{plain}

\newtheorem{thm}{Theorem~}[section] 
\newtheorem{lem}[thm]{Lemma~}

\newtheorem{prop}[thm]{Proposition~}

\theoremstyle{remark}
\newtheorem{rmk}[thm]{Remark~}

\theoremstyle{definition}
\newtheorem{defn}[thm]{Definition~}

\newcommand{\calH}{\mathcal{H}}

\newcommand{\calF}{\mathcal{F}}

\newcommand{\CC}{\mathbb{C}}
\newcommand{\ZZ}{\mathbb{Z}}
\newcommand{\RR}{\mathbb{R}}
\newcommand{\LL}{\mathbb{L}}
\newcommand{\PP}{\mathbb{P}}
\newcommand{\FF}{\mathbb{F}}
\newcommand{\QQ}{\mathbb{Q}}

\newcommand{\DD}{\mathbb{D}}
\newcommand{\HH}{\mathbb{H}}

\newcommand\PGL{\mathrm{PGL}}
\newcommand\SO{\mathrm{SO}}
\newcommand\PSL{\mathrm{PSL}}

\newcommand\id{\mathrm{id}}

\newcommand\rank{\mathrm{rank}}
\newcommand\I{\mathrm{I}}

\newcommand\SL{\mathrm{SL}}

\newcommand\diag{\mathrm{diag}}
\newcommand\GL{\mathrm{GL}}

\newcommand\rO{\mathrm{O}}
\newcommand{\Aut}{\mathrm{Aut}}

\newcommand{\bs}{\backslash}
\newcommand{\dbs}{\bs\hspace{-0.5mm}\bs}

\title{Cubic Fourfolds with an Order-$7$ Automorphism}

 \author[X. He]{Xuancong He}
\address{Fudan University, China}
\email{xche25@m.fudan.edu.cn}

 \author[Y. Li]{Yi Li}
\address{Tsinghua University, China}
\email{yi-li22@mails.tsinghua.edu.cn}

 \author[S. Wang]{Shihao Wang}
\address{Tsinghua University, China}
\email{wangshihao25@mails.tsinghua.edu.cn}

 \author[Z. Zheng]{Zhiwei Zheng}
\address{Tsinghua University, China}
\email{zhengzhiwei@mail.tsinghua.edu.cn}

\date{}

\begin{document}
\bibliographystyle{amsalpha}

\begin{abstract}
We study smooth cubic fourfolds admitting an automorphism of order $7$. It is known that the possible symplectic automorphism groups of such cubic fourfolds are precisely $F_{21}$, $\mathrm{PSL}(2,\mathbb{F}_7)$, and $A_7$. In this paper, we determine all possible full automorphism groups of smooth cubic fourfolds with an automorphism of order $7$. We also investigate the moduli spaces of cubic fourfolds whose automorphism group is either $F_{21}$ or $\mathrm{PSL}(2,\mathbb{F}_7)$, describing them both as GIT quotients and as locally symmetric varieties. In particular, we give an explicit description of the singular cubic fourfolds that appear in the boundary of the corresponding GIT quotients. For these two cases, we determine the commensurability classes of the monodromy groups by explicitly identifying certain arithmetic subgroups. As an interesting consequence, we prove that the period domain for cubic fourfolds equipped with an order-$7$ automorphism is isogenous to a Hilbert modular surface.
\end{abstract}

\maketitle

 \setcounter{tocdepth}{1}
	\tableofcontents

\emph{Notation}:
\begin{enumerate}
\item For two groups $G_1, G_2$, we write $G_1<G_2$ if $G_1$ is a subgroup of $G_2$.
\item We use $L_2(7)$ to denote $\PSL_2(\FF_7)$. 
\item For $\frac{1}{n}(m_{1},\dots,m_{k})$, we mean the diagonal matrix $\diag(\zeta_{n}^{m_1},\dots,\zeta_{n}^{m_{k}})$, where $\zeta_{n}=\text{exp}(\frac{2\pi \sqrt{-1}}{n})$ and $m_{1},\dots,m_{k},n\in \mathbb{Z}$. 
\item For $g\in \GL(n,\CC)$, denote by $\overline{g}$ the image of $g$ in $\PGL(n,\CC)$ under the natural map $\GL(n,\CC)\rightarrow \PGL(n,\CC)$ defined by quotient the center of $\GL(n,\CC)$. 
\item For permutation $\sigma \in S_n$, we also use $\sigma$ to denote the matrix $A_\sigma \in \GL(n,\CC)$ such that $(e_{\sigma(1)},..,e_{\sigma(n)}) = (e_1,..,e_n)A_\sigma$. For example, we use $(123) \in S_3$ to denote $\begin{pmatrix}
    0&0&1\\
    1&0&0\\
    0&1&0
\end{pmatrix}.$
\end{enumerate}

\section{Introduction}
\label{introduction}

Cubic fourfolds play an important role in complex algebraic varieties thanks to their relation to hyperk\"ahler geometry and rationality problems. What makes cubic fourfolds special is that their middle cohomology has Hodge numbers $0,1,21,1,0$, which is very similar to the middle cohomology of K3 surfaces. Therefore, the moduli space of cubic fourfolds shares very similar properties with the moduli space of K3 surfaces and they both have corresponding Torelli theorems. Historically, Voisin \cite{voisin1986torelli, voisin2008erratum} proved a global torelli theorem for cubic fourfolds. An alternative proof was provided by Looijenga \cite{looi2009period}. Hassett \cite{Hassett_2000} conjectured that the image of the period map is the complement is two specific Noether-Lefschetz divisors. This conjecture was independently confirmed by Looijenga \cite{looi2009period} and Laza \cite{laza2010period}, who also identified the GIT compactification with the Looijenga compactification (with respect to one of the Noether-Lefschetz divisors) for the moduli space of cubic fourfolds.

In \cite{yu2020moduli}, Yu--Zheng studied moduli space of cubic fourfolds with specified group action, and identified the GIT compactification of the moduli space to the Looijenga compactification of the corresponding arithmetic quotient of local period domain. In this paper, we follow the work of \cite{yu2020moduli}, and focus on the study of smooth cubic fourfolds with an automorphism of order $7$. According to \cite[Theorem 3.8]{gonzalez2011automorphisms} (see also \cite[Theorem 1.2 (7b)]{laza2022automorphisms}), up to a linear coordinate transformation, any such cubic fourfold can be defined by an equation:
\[
F_{a,b} = x_1^2x_2+x_3^2x_4+x_5^2x_6+x_2^2x_3+x_4^2x_5+x_6^2x_1+ax_1x_3x_5+bx_2x_4x_6,
\]
and the order-$7$ automorphism is given by $g_7 = \frac{1}{7}(1,5,4,6,2,3)$. An automorphism of a smooth cubic fourfold is called symplectic if its action on $H^{3,1}$ is trivial. From the equation $F_{a,b}$, it is direct to see that any order-$7$ automorphism of a smooth cubic fourfold is symplectic.


In \cite{laza2022automorphisms}, Laza--Zheng classified all symplectic automorphism groups of smooth cubic fourfolds. From \cite[Theorem 1.2, (7b), (8), (9)]{laza2022automorphisms}, we have 
\begin{thm}[Laza--Zheng]
\label{theorem: laza--zheng}
If a smooth cubic fourfold $X$ admits an automorphism of order $7$, then $\Aut^s(X)=F_{21}, L_2(7)$ or $A_7$. Moreover, cubic fourfolds with $F_{21} < \Aut^s(X)$ form a $2$-dimensional family. Those with $L_2(7) < \Aut^s(X)$ form a $1$-dimensional subfamily of it and those with $A_7 < \Aut^s(X)$ are two isolated points. 
\end{thm}






However, a full classification of automorphism groups of cubic fourfolds with order-$7$ actions has not been obtained yet. There is also a lack of detailed study of their moduli. In this paper, we give a comprehensive investigation of such cubic fourfolds, from both geometric and arithmetic perspectives.

Next we give the main theorem (see \S \ref{section:full automorphism group}) of this paper. Note that we borrow many explicit equations from previous work including \cite{gonzalez2011automorphisms}, \cite{fu2016classification},  \cite{hohn2019finite}, \cite{laza2022automorphisms}, \cite{yang2024automorphism}, \cite{KOIKE202512}.
\begin{thm}
\label{theorem: main}
If a smooth cubic fourfold $X$ admits an automorphism of order $7$, then $\Aut(X) = F_{21}, F_{21}\rtimes C_2,F_{21}\rtimes C_6,L_2(7),L_2(7)\rtimes C_2,  A_7$ or $S_7$. More precisely, the pair $(\Aut^s(X), \Aut(X))$ equals to exactly one of the following cases:
    \begin{enumerate}[(1)]
        \item $(F_{21},F_{21})$. This is the generic case. And defining equation of $X$ can be written as 
        \[        x_{1}^{2}x_{2}+x_{3}^{2}x_{4}+x_{5}^{2}x_{6}+x_{2}^{2}x_{3}+x_{4}^{2}x_{5}+x_{6}^{2}x_{1}+ax_{1}x_{3}x_{5}+bx_{2}x_{4}x_{6},
        \]
        with $a,b\in \CC$ distinct. 
        \item $(F_{21},F_{21}\rtimes C_{2})$. This is a 1-dimensional family in $\overline{\mathcal{F}}_{F_{21}}$. In this case, defining equation of $X$ can be written as 
        \[        x_{1}^{2}x_{2}+x_{3}^{2}x_{4}+x_{5}^{2}x_{6}+x_{2}^{2}x_{3}+x_{4}^{2}x_{5}+x_{6}^{2}x_{1}+a(x_{1}x_{3}x_{5}+x_{2}x_{4}x_{6}),
        \]
         where $a\in \mathbb{C}^\times$.
        \item $(F_{21}, F_{21}\rtimes C_{6})$. In this case, $X$ is isomorphic to the Klein cubic fourfold, defined by the equation 
        \[        x_{1}^{2}x_{2}+x_{3}^{2}x_{4}+x_{5}^{2}x_{6}+x_{2}^{2}x_{3}+x_{4}^{2}x_{5}+x_{6}^{2}x_{1}.
        \] 
        See also \cite[Example 6.1 (9)]{yang2024automorphism}.
        \item $(L_{2}(7),L_{2}(7))$. It is the generic case for the family $\overline{\mathcal{F}}_{L_{2}(7)}$, which is a 1-dimensional subfamily in $\overline{\mathcal{F}}_{F_{21}}$. The defining equations are given by
        \[
        c_{1}(\sum\limits _{i=1}^{6}x_{i}^{3}-(\sum\limits_{i=1}^{6}x_{i})^{3})+c_{2}(x_{1}x_{2}x_{4}+x_{2}x_{3}x_{5}+x_{3}x_{4}x_{6}-(x_{1}x_{3}+x_{4}x_{5}+x_{2}x_{6})(\sum\limits_{i=1}^{6}x_{i})).
        \]
        \item $(L_{2}(7),L_{2}(7)\rtimes C_{2})$. There exists and only exists two $X$ up to isomorphism. They are Galois conjugate to each other and the defining equations can be given by 
        \[
        \sum\limits _{i=1}^{6}x_{i}^{3}-(\sum\limits_{i=1}^{6}x_{i})^{3}\pm\frac{3\sqrt{2}}{2}(x_{1}x_{2}x_{4}+x_{2}x_{3}x_{5}+x_{3}x_{4}x_{6}-(x_{1}x_{3}+x_{4}x_{5}+x_{2}x_{6})(\sum\limits_{i=1}^{6}x_{i})).
        \]
        $($Equation of one of the examples for another coordinate was given by \cite[Example 6.1 (13)]{yang2024automorphism}.$)$
        \item $(\text{A}_{7},\text{A}_{7})$. We only have 1 possibility for $X$ up to isomorphism, explicit defining equations given by \[x_1^3+x_2^3+x_3^3+\frac{12}{5}x_1x_2x_3+x_1x_4^2+x_2x_5^2+x_3x_6^2+\frac{4\sqrt{15}}{9}x_4x_5x_6.
        \]
        This equation was first obtained by \cite[Theorem 6.14]{yang2024automorphism}. See \cite[Page 16 (2.3)]{KOIKE202512} for an equation up to another coordinate.
        \item $(\text{A}_{7},\text{S}_{7})$. We only have 1 possibility for $X$ up to isomorphism, with defining equation 
        \[
        \sum \limits_{i=1}^{6}x_{i}^{3}-(\sum\limits_{i=1}^{6}x_{i})^{3}.
        \]
        see also \cite[Table 11]{hohn2019finite}.
    \end{enumerate}
\end{thm}

Along the proof of Theorem \ref{theorem: main}, we obtain a clear understanding of the stable and semistable locus, see Proposition \ref{C7GITstb} and Proposition \ref{F21GIT} in \S\ref{subsection: GIT moduli of F21}.

We know from \cite{yu2020moduli} that there is an isomorphism 
$\calF \cong (\DD_T\setminus\calH_{s})/\Gamma_T$, where $\calF$ is the moduli space of smooth cubic fourfolds with certain symmetries, $\DD_T$ is the corresponding local period domain, $\calH_{s}$ is a specific hyperplane arrangement in $\mathbb{D}_{T}$, and $\Gamma_T$ is an arithmetic subgroup acting on $\DD_T$. For details of these notations, see \S \ref{preliminary}.


Using the $a,b$ model $F_{a,b}$, we calculate the locus of singular cubic fourfolds and obtain the following results, see \S \ref{subsection: singular cubic fourfolds} and Proposition \ref{singuL27}.

\begin{prop}
    For cubic fourfolds inside the family $F_{a,b}$, there are three curves corresponding to singular cubic fourfolds and they only have isolated singularities of $A_1$ or $A_2$ type. Moreover, there is an action of the cyclic group $C_{6}$ on the parameter space $\mathbb{A}^{2}$ for the space $F_{a,b}$ such that $ C_{6}\backslash\mathbb{A}^2 \hookrightarrow \overline{\calF}_{F_{21}}$ is a Zariski open subset. As a corollary, $\overline{\calF}_{F_{21}}$ is a rational surface. 
\end{prop}
\begin{prop}
\label{proposition: L_2(7) moduli}
    The moduli of smooth cubic fourfolds with $L_2(7) < \Aut^s(X)$, denoted by $\calF_{L_2(7)}$, has compactification $\overline{\calF}_{L_2(7)} \cong \PP^1$. The boundary $\overline{\calF}_{L_2(7)} \setminus \calF_{L_2(7)}$ has exactly one determinantal cubic fourfold, one cubic fourfold with $7$ nodes and one cubic fourfold with $14$ nodes. 
\end{prop}


Via period maps, the moduli spaces $\mathcal{F}_{F_{21}}$ and $\mathcal{F}_{L_2(7)}$ are Zariski open subsets of the corresponding arithmetic quotients of type IV domains. We investigate these two arithmetic quotients via tube domain models. 

Let $\DD_{T_1}$ be the local period domain of the $L_2(7)$-family and $\DD_{T_2}$ be the local period domain of the $F_{21}$-family. Firstly, we compute the Gram matrices for $T_{1}$ and $T_{2}$ respectively; see \S \ref{conslatL27} and \S \ref{conslatF21}. 

The Gram matrix of $T_1$ is $\diag (\begin{bmatrix} -2 & 1 \\ 1 & 10 \end{bmatrix}, -28)$ up to an integral basis. By elementary number theory, one can check that this lattice has no nonzero isotropic vector. This implies the following:
\begin{prop}[= Proposition \ref{proposition:cocompactness}]
\label{proposition: quotient compact}
    The quotient curve $\Gamma_{T_1}\backslash\DD_{T_1}$ is compact. In particular, it is not isogenous to $\PSL(2,\ZZ)\backslash\HH$.  
\end{prop}
The compactness of $\Gamma_{T_1}\backslash\DD_{T_1}$ can also be deduced from Proposition \ref{proposition: L_2(7) moduli} using period map, see Remark \ref{rmk:compactnessfromGIT} for an argument.  


In fact, we can determine what $\Gamma_{T_1}\backslash\DD_{T_1}$ is isogenous to by computing $\Gamma_{T_1}$ as follows. There is a natural surjection $\Phi_{1}\colon\SL(2,\RR)\rightarrow \text{SO}^{+}(1,2)$, see \S \ref{tubedomainL27} for the detailed construction. Using this surjection, we have the following description of $\Gamma_{T_1}$:

\begin{prop}[= Proposition \ref{proposition:arithgroupofL27}]
    Denote $\Gamma^{\prime \prime} = \Phi_{1}^{-1}(\rO^+(T_{1})\cap \SO^+(1,2))$. We have $\Phi_{1}(\Gamma^{\prime\prime})\cong \Gamma_{T_1}$ and the following description:
\begin{align*}
    \Gamma^{\prime \prime} = &\{
    \begin{pmatrix}
        \frac{u\sqrt{u^\prime}+v\sqrt{v^\prime}}{2} & \frac{w\sqrt{w^\prime}+x\sqrt{x^\prime}}{2} \\
        \frac{w\sqrt{w^\prime}-x\sqrt{x^\prime}}{2} &
        \frac{u\sqrt{u^\prime}-v\sqrt{v^\prime}}{2}
    \end{pmatrix}|u,v,w,x \in \ZZ, (u^\prime, v^\prime, w^\prime, x^\prime) = (1,21,6,14)\  \text{or} \ (2,42,3,7)\\ &\text{up to permutation by }\mathrm{V} < \mathrm{S}_4, \text{and }u^2u^\prime - v^2v^\prime - w^2w^\prime + x^2x^\prime = 4\}
\end{align*}
(Here we use $\mathrm{V}$ to denote the Klein $4$-group generated by $(12)(34)$ and $(13)(24)$.)
\end{prop}

We determine the type of $\Gamma^{\prime \prime}$ as an arithmetic subgroup of $\SL(2,\RR)$ up to commensurability and conjugacy and obtain the following (the notation follows \cite{morris2015introduction}):
\begin{thm}[= Theorem \ref{theorem:typeofdomainL27}]
    $\Gamma_{T_1}\backslash\DD_{T_1}$ is isogenous to $\SL(1, \HH_\ZZ^{21,6})\backslash\HH$. 
\end{thm}

Similarly, we describe $\Gamma_{T_{2}}$ up to a finite index. There is a natural surjection $\Phi_{2}\colon\SL(2,\RR)\times \SL(2,\RR)\rightarrow \text{SO}^{+}(2,2)$, see \S \ref{tubedomainL27} for the detailed construction. Using this isomorphism, we have the following result:

\begin{prop}[= Proposition \ref{strulem}]
\label{descGammaT2}
    We have that $\Gamma_{T_{2}}$ is an index-2 subgroup of $\rO^{+}(T_{2})$ and $\rO^{+}(T_{2})=\langle -\Id, \Gamma_{T_{2}}\rangle$.

    On the other hand, $\rO^{+}(T_{2})=\langle\SO^{+}(T_{2}),P\rangle$, where $P$ is an involution constructed in \S \ref{conslatF21}.

    Denote $\Gamma^\prime = \Phi_2^{-1}(\SO^+(T_{2}))$, then we have $h \in \Gamma^\prime$ if and only if
    \[
    h =(\begin{pmatrix}
        \frac{\alpha_1\sqrt{\alpha}+\beta_1\sqrt{\beta}}{2}
        &\frac{1}{\sqrt{7}}\frac{\gamma_1\sqrt{\gamma}+\delta_1\sqrt{\delta}}{2}\\
        \sqrt{7}\frac{\gamma_2\sqrt{\gamma}-\delta_2\sqrt{\delta}}{2}&\frac{\alpha_2\sqrt{\alpha}-\beta_2\sqrt{\beta}}{2}
    \end{pmatrix},
    \begin{pmatrix}
        \frac{\alpha_2\sqrt{\alpha}+\beta_2\sqrt{\beta}}{2}
        &\sqrt{7}\frac{\gamma_2\sqrt{\gamma}+\delta_2\sqrt{\delta}}{2}\\
        \frac{1}{\sqrt{7}}\frac{\gamma_1\sqrt{\gamma}-\delta_1\sqrt{\delta}}{2}&\frac{\alpha_1\sqrt{\alpha}-\beta_1\sqrt{\beta}}{2}
    \end{pmatrix})
    \]
    where 
\begin{itemize}
    \item $(\alpha,\beta,\gamma,\delta) = (1,21,3,7)$ up to permutation by $\mathrm{V} < \mathrm{S}_4$
    \item $\alpha_i,\beta_i,\gamma_i,\delta_i \in \ZZ$ such that $\alpha_i-\beta_i,\gamma_i-\delta_i\in 2\ZZ$
    \item $\alpha_1\alpha_2\alpha - \beta_1\beta_2\beta-\gamma_1\gamma_2\gamma+\delta_1\delta_2\delta = 4$ and $\alpha_1\beta_2-\alpha_2\beta_1=\gamma_1\delta_2-\gamma_2\delta_1$. 
\end{itemize}
    (The last condition is equivalent to saying that the determinants of $h_1$ and $h_2$ are $1$.)
\end{prop}

We will show that $\Gamma^{\prime}$ contains a Hilbert modular group as an index-4 subgroup, in particular, the modular surface admits a degree-4 covering from the Hilbert modular surface. This is stated as the following theorem.

\begin{thm}
\label{theorem: hilbert modular group}
    There is a subgroup $H<\Gamma^{\prime}$ such that $\Gamma^{\prime}/H\cong \ZZ/2\ZZ\oplus \ZZ/2\ZZ$.

    There is a natural homomorphism $\psi \circ \varphi\colon \SL(2,F)\to \SL(2,\RR)\times\SL(2,\RR)$ mapping $\Gamma(\mathcal{O}_F\oplus\mathfrak{p})$ isomorphically onto $H$, where $F=\mathbb{Q}(\sqrt{21})$ and $\mathfrak{p} = (\frac{7+\sqrt{21}}{2})$, the prime ideal of $\mathcal{O}_F$ lying over $7$.

    In particular, there are natural morphisms
    \[   
    \Gamma(\mathcal{O}_F\oplus\mathfrak{p})\backslash\HH^2 \to \Gamma^\prime\backslash\HH^2 \xrightarrow{\sim} \SO^+(T_{2})\backslash\mathbb{D}_{T_{2}} \to \Gamma_{T_{2}}\backslash\mathbb{D}_{T_{2}}\cong\rO^+(T_{2})\backslash\DD_{T_{2}}.
    \]
    Each of them is of degree $2$, except for the isomorphisms.
\end{thm}

\begin{rmk}
In \cite{dolgachev2018geometry} and \cite{farb2021arithmeticity}, the authors studied the geometry and monodromy group of the Wiman-Edge pencil. They proved that the monodromy group is commensurable to the Hilbert modular group $\SL(2,\ZZ[\sqrt{5}])$. In their case, the period map is defined via Jacobian of curves, and they used Picard--Lefschetz theory to calculate the monodromy group. In this paper (Theorem \ref{theorem: hilbert modular group}), by considering the Hodge structures on the middle cohomology of cubic fourfolds, we provide another monodromy group commensurable to Hilbert modular group. 
\end{rmk}

\textbf{Acknowledgement}: This work was carried out as a project in the 2025 summer school “Algebra and Number Theory,” jointly organized by Peking University and the Academy of Mathematics and Systems Science, Chinese Academy of Sciences. We gratefully acknowledge Professor Shou-Wu Zhang for initiating this summer school, and we thank all the organizers for their dedicated efforts. We thank Chenglong Yu for stimulating discussion on commensurabilities among arithmetic groups. We thank Liang Xiao and Kaiyuan Gu for their interest and helpful discussion on Shimura varieties. We thank Eduard Looijenga and Gerard van der Geer for their interest and helpful comments on an earlier version of this paper. The author Z. Zheng is supported by NSFC 12301058.

\section{Preliminary: Cubic Fourfolds and Periods}
\label{preliminary}
\subsection{Periods for Cubic Fourfolds}
\label{periodsfourfolds}
First we recall the global Torelli theorem for cubic fourfolds. It allows us to connect the GIT moduli space of cubic fourfolds to the moduli of their Hodge structures. 

Let $X$ be a smooth cubic fourfold, the middle cohomology $H^4(X, \mathbb{Z})$, with the natural intersection pairing, is the unimodular odd lattice $\Lambda$ of signature $(21, 2)$. Let $\eta_X \in H^4(X, \mathbb{Z})$ be the square of the hyperplane class of $X$, we can choose an isomorphism of lattices $\Lambda \cong \mathrm{E}_8^{\oplus2}\oplus \mathrm{U}^{\oplus2}\oplus \mathrm{I}_{3,0}$ such that $\eta_X = (1,1,1) \in I_{3,0}$. The primitive cohomology $H^4(X, \mathbb{Z})_{prim} = \langle\eta_X\rangle^\perp$ carries a Hodge structure of K3 type and let $\Lambda_0$ denote its lattice type, we have $\Lambda_0 \cong \mathrm{E}_8^{\oplus 2}\oplus \mathrm{U}^{\oplus 2} \oplus \mathrm{A}_2$. The local period domain for Hodge structures on $\Lambda_0$ is the $20$-dimensional Type IV domain
\[
\mathbb{D} = \{x \in \mathbb{P}(\Lambda_0 \otimes \mathbb{C})|(x,x) = 0, (x, \overline{x})<0\}^+
\]
(where the subscript $+$ means taking one of the two connected components). 

Let $\mathbb{P}(H^0(\mathbb{P}^5, \mathcal{O}(3)))^{sm}$ be the subset of $\mathbb{P}(H^0(\mathbb{P}^5, \mathcal{O}(3)))$ consisting of smooth cubic fourfolds. It is known that any isomorphism between two smooth cubic fourfolds can be lifted to a linear isomorphism. Therefore, the moduli of smooth cubic fourfolds is $\mathcal{M} = \SL(6, \mathbb{C})\backslash\mathbb{P}(H^0(\mathbb{P}^5, \mathcal{O}(3)))^{sm}$ with GIT compactification $\overline{\mathcal{M}} = \SL(6, \mathbb{C})\dbs\mathbb{P}(H^0(\mathbb{P}^5, \mathcal{O}(3)))$. 

Let $\rO(\Lambda_0)$ denote the automorphism group of the lattice $\Lambda_0$. For any $g \in \rO(\Lambda_0)$, we use $\overline{g}$ to denote its image in the orthogonal group of the discriminant group $\rO(A_{\Lambda_0})$. The global monodromy group for the family $\mathbb{P}(H^0(\mathbb{P}^5, \mathcal{O}(3)))^{sm}$ is (cf. \cite[Theorem 2]{be1986})
\[\widehat{\Gamma} = \{g \in \rO(\Lambda_0)|\overline{g} = Id_{A_{\Lambda_0}}, g(\mathbb{D}) = \mathbb{D}\}\]

Therefore, by associating a cubic fourfold with its Hodge structure, one obtains the global period map
\[
\mathscr{P}\colon \mathcal{M} \to \widehat{\Gamma}\backslash\mathbb{D}
\]
which is injective due to Voisin \cite{voisin1986torelli, voisin2008erratum}, see also Looijenga \cite{looi2009period}. Here the analytic variety $\widehat{\Gamma}\backslash\mathbb{D}$ is in fact a quasi-projective variety due to Baily-Borel compactification \cite{baily1966compactification}. 

\begin{defn}
\
    \begin{itemize}
        \item A norm $2$ vector $v$ in $\Lambda_0$ is called a short root. The orthogonal complements of each short root give a $\widehat{\Gamma}$-invariant hyperplane arrangements $\mathcal{H}_6$ in $\mathbb{D}$. The associated Heegner divisor is $\mathcal{C}_6 =\widehat{\Gamma}\backslash \mathcal{H}_6$. 
        \item A norm $6$ vector $v$ in $\Lambda_0$ with divisibility $3$ is called a long root. The orthogonal complements of each long root give a $\widehat{\Gamma}$-invariant hyperplane arrangements $\mathcal{H}_2$ in $\mathbb{D}$. The associated Heegner divisor is $\mathcal{C}_2 = \widehat{\Gamma}\backslash\mathcal{H}_2$. 
    \end{itemize}
\end{defn}
Both $\mathcal{C}_6$ and $\mathcal{C}_2$ are irreducible divisors \cite{Hassett_2000} in $\widehat{\Gamma}\backslash\mathbb{D}$. Let $\mathcal{M}^{\mathrm{ADE}}$ denote the moduli space of cubic fourfolds with at most ADE singularities, we state the following theorem due to \cite{looi2009period} and \cite{laza2010period}. 
\begin{thm}
    The period map for cubic fourfolds gives an isomorphism of quasi-projective varieties
    \[
    \mathscr{P}\colon \mathcal{M} \to \widehat{\Gamma}\backslash(\mathbb{D}-(\mathcal{H}_2 \cup \mathcal{H}_6))
    \]
    which extends to an isomorphism
    \[
    \mathscr{P}\colon \mathcal{M}^{\mathrm{ADE}} \to \widehat{\Gamma}\backslash(\mathbb{D}- \mathcal{H}_2 )
    \] 
    and further to an isomorphism
    \[
    \mathscr{P}\colon \overline{\mathcal{M}}\to \overline{\widehat{\Gamma}\backslash\DD}^{\mathcal{H}_2}
    \]
    where $\overline{\widehat{\Gamma}\backslash\DD}^{\mathcal{H}_2}$ is the Looijenga compactification. 
\end{thm}
\subsection{Moduli and Periods for Cubic Fourfolds with Symmetries}
\label{subsection:moduli}
Next we recall Yu--Zheng's work \cite{yu2020moduli} on moduli space of cubic fourfolds with a specified group action. 

Let $\CC^{6}$ be a complex vector space of dimension $6$. The space $H^0(\mathbb{P}(\CC^{6}), \mathcal{O}(3))$ is the space of degree $3$ polynomials on $\CC^{6}$ and we have a natural action of $\SL(6,\CC)$ on $H^0(\mathbb{P}(\CC^{6}), \mathcal{O}(3))$. For each matrix $A \in \SL(6,\CC)$ and $F \in H^0(\mathbb{P}(\CC^{6}), \mathcal{O}(3))$, we denote $(AF)(\vec{x}) = F(\vec{x}A)$. 

The center of $\SL(6,\CC)$ is the group $\mu_6$ consisting of $6$-th roots of unity. Let $\widetilde{G}$ be a finite subgroup of $\SL(6,\CC)$ containing $\mu_6$ and $G = \widetilde{G}/\mu_6$ its image in $\PGL(6,\CC)$. Let $\lambda\colon \widetilde{G} \to \mathbb{C}^\times$ be a character of $\widetilde{G}$ such that $\lambda|_{\mu_6}$ sends $\xi \in \mu_6$ to $\xi^{-3}$ and let $\mathcal{V}_{\widetilde{G},\lambda}$ be the $\lambda$-eigenspace of $H^0(\mathbb{P}(\CC^{6}), \mathcal{O}(3))$ with respect to $G$. Geometrically,  $\mathcal{V}_{\widetilde{G},\lambda}$ contains some of the cubic fourfolds with automorphism group containing $G$. Now let
\[
N = \{A \in \SL(6,\CC)|A\widetilde{G}A^{-1} = \widetilde{G}, \lambda(AgA^{-1}) = \lambda(g), \forall g \in \widetilde{G}\}.
\]
We know that $N$ is a reductive subgroup of $\SL(6,\CC)$ and has an action on $\mathcal{V}_{\widetilde{G},\lambda}$. Let $\mathcal{V}_{\widetilde{G},\lambda}^{sm}$ be the subspace containing smooth cubic fourfolds and $\mathcal{V}_{\widetilde{G},\lambda}^{ss}$ the subspace containing semistable points with respect to the action of $N$. Respectively, we let $\mathcal{F} = N \dbs \mathcal{V}_{\widetilde{G},\lambda}^{sm}$ and $\overline{\mathcal{F}} = N \dbs \mathcal{V}_{\widetilde{G},\lambda}^{ss}$ be the corresponding GIT quotients. We borrow the following theorem form Yu--Zheng's work
\begin{thm}
\label{yz}
    There is a natural morphism $j\colon \overline{\mathcal{F}} \to \overline{\mathcal{M}}$ sending $[F] \in \mathcal{F}$ to $[F] \in \mathcal{M}$ for any $F \in \mathcal{V}_{\widetilde{G},\lambda}^{sm}$. This morphism is finite. Furthermore, if there is some point $F \in \mathcal{V}_{\widetilde{G},\lambda}^{sm}$ with automorphism group $\Aut(Z(F)) = G$, then this morphism is a normalization of its image. (Here we denote $Z(F)$ for its corresponding cubic fourfold). 
\end{thm}

Now let $X = Z(F)$ be a generic point $F \in \mathcal{V}_{\widetilde{G},\lambda}^{sm}$, the action of $G$ on $X$ induces an action of $G$ on $\Lambda = H^4(X, \ZZ)$ and therefore on $\Lambda_0$ since it acts on $\eta_X$ trivially. Meanwhile, $G$ acts on $H^{3,1}(X)$ and this action is a character $\zeta\colon G \to \mathbb{C}^{\times}$. We only consider $\zeta = \overline{\zeta}$ here. Therefore, according to Yu--Zheng's work, we define $T = (\Lambda_0)_\zeta = \{x \in \Lambda_0|gx=\zeta(g)x,\forall g \in G\}$ which is a lattice of signature $(n,2)$ with $n = \dim \overline{\mathcal{F}}$. And we define the local period domain $\mathbb{D}_T = \{x \in \PP( T\otimes \CC)|(x,x)=0,(x,\overline{x}) <0\}^+$. Let $\overline{G} < \widehat{\Gamma}$ denote the image of $G$ in $\widehat{\Gamma}$, and let $\Gamma_T$ be the normalizer of $\overline{G}$ in $\widehat{\Gamma}$. We borrow the following theorem from Yu--Zheng's work \cite{yu2020moduli}:
\begin{thm}
\label{thm:Yu--Zhengperiodmapsymmcubic4fold}
    There is a global period map 
    \[\mathscr{P}\colon\mathcal{F} \to \Gamma_T\backslash(\mathbb{D}_T-\mathcal{H}_s)\] 
    which is an algebraic isomorphism. 
    Here $\mathcal{H}_s = \mathbb{D}_T\cap(\mathcal{H}_2\cup\mathcal{H}_6)$. Moreover, let $\mathcal{F}^{\mathrm{ADE}}$ denote the open subset parametrizing those with at most ADE singularities, then this period map extends to isomorphisms
     \[
    \mathscr{P}\colon\mathcal{F}^{\mathrm{ADE}} \to \Gamma_T\backslash(\mathbb{D}_T- \mathcal{H}_*)
     \]
     $(\mathcal{H}_* = \DD_T\cap\mathcal{H}_2)$
     and
     \[
     \mathscr{P}\colon \overline{\mathcal{F}} \to \overline{\Gamma_T\backslash\DD_T}^{\mathcal{H}_*}
     \]
     where $\overline{\Gamma_T\backslash\DD_T}^{\mathcal{H}_*}$ is the Looijenga compactification.
\end{thm}

\section{GIT Analysis for Cubic Fourfolds Admitting Order-$7$ Action}
\label{section: GIT moduli for 7 action}
In this section, we investigate the geometry of the GIT compactification of the moduli of cubic fourfolds with order-$7$ automorphism. We obtain a complete description of the singular cubic fourfolds inside it and their singularities. 

\subsection{Some Preliminery Group-theoretic Analysis}
By \S \ref{periodsfourfolds}, to construct the GIT moduli of smooth cubic fourfolds with order-7 action, we need to construct $N \dbs \PP\mathcal{V}_{T}$ for $T=(\widetilde{G},\lambda)$. However, different projective representations of $G$ and different characters $\lambda$ may give different families. 

The following lemma shows that there is no such ambiguity. More precisely, we only need to consider one projective representation for each $G=C_{7},F_{21}\ \text{or} \ L_{2}(7)$. In fact, we only need to deal with one linear lifting for this representation and the $G$-invariant cubic polynomials with respect to this linear lifting.


Fix generators $g_3$ and $g_7$ of $F_{21}$ such that $g_3$ has order $3$, $g_7$ has order $7$ and $g_{3}g_{7}g_{3}^{-1}=g_{7}^{2}$. We identify $C_7$ with the subgroup $\langle g_7\rangle$ of $F_{21}$.

Consider 6-dimensional representations of $C_{7}$ and $F_{21}$ determined by sending $g_{3}$ to $(e_{i}\mapsto e_{i+2})$, where $(e_{1},e_{2},\dots,e_{6})$ represents the standard basis of $\CC^{6}$, and sending $g_{7}$ to $\frac{1}{7}(1,5,4,6,2,3)$.

\begin{lem}
\label{repinv}
    For $G=C_{7}, F_{21}\ \text{or} \ L_{2}(7)$, if there exists a smooth cubic fourfold $X$ such that $G<\Aut^{s}(X)$ where $G$ acts on $X$ via some faithful projective representation $\rho\colon G\rightarrow \PGL(6,\mathbb{C})$. Then $\rho$ is uniquely determined up to isomorphism and it can be lifted to a linear representation $\widetilde{\rho}\colon G\rightarrow \GL(6,\mathbb{C})$.

    For $C_{7}$ and $F_{21}$, $\widetilde{\rho}$ can be chosen as we describe above. For $L_{2}(7)$, see \S \ref{subsecL27}.
    
    Moreover, for all smooth cubic fourfold $X$ with $G<\Aut^{s}(X)$, there exists a coordinate change such that $X$ is defined by a $\widetilde{\rho}(G)$-invariant polynomial.
\end{lem}

\begin{proof}
(1) The case that $G=C_{7}$ directly follows from \cite[Theorem 1.1]{fu2016classification}. Any smooth cubic fourfold with $C_{7}$-action can be defined by the polynomial $x_{1}^{2}x_{2}+x_{3}^{2}x_{4}+x_{5}^{2}x_{6}+x_{2}^{2}x_{3}+x_{4}^{2}x_{5}+x_{6}^{2}x_{1}+ax_{1}x_{3}x_{5}+bx_{2}x_{4}x_{6}$ with $a,b\in \CC$ after a coordinate changing.

(2) For $G=F_{21}$, each projective representation $\rho$ of $G$ can be lifted to be a linear representation $\widetilde{\rho}$, see \cite[Page 4]{KOIKE202512}. By \cite[Theorem 1.1]{fu2016classification} again, we may assume $\widetilde{\rho}$ is what we describe above. Such 6-dimensional representation of $F_{21}$ is unique. We have $\text{det}(\widetilde{\rho}(g_{3}))=1$. Suppose $\widetilde{\rho}(g_{3})F=\lambda F$, then $\lambda^{3}=1$. By \cite[Lemma 3.2]{fu2016classification} or \cite[Lemma 6.10]{yang2024automorphism}, $1=\text{det}(\widetilde{\rho}(g_{3}))=\lambda^{2}$. So $\lambda=1$ and $F$ is $F_{21}$-invariant for representation $\widetilde{\rho}$.

(3) For $G=L_{2}(7)$, if there is a smooth cubic fourfold $X$ with $\Aut^{s}(X)\cong L_{2}(7)$ with $L_{2}(7)$ acting on $X$ via some projective representation $\rho$, then $\rho$ can be lifted to a linear representation $\widetilde{\rho}$, see \cite[Lemma 2.5]{KOIKE202512}. By \cite[Page 5]{KOIKE202512}, let $g_{4}\in L_{2}(7)$ be an order-$4$ element, then $\text{Tr}(\widetilde{\rho}(g_{4}))=0$. There is only one $6$-dimensional representation of $L_{2}(7)$ satisfying above condition, the irreducible one. Thus $\rho$ is uniquely determined. And defining equation $F$ of $X$ must be $L_{2}(7)$-invariant for $\widetilde{\rho}$ since $L_{2}(7)$ is a simple group.
\end{proof}

From now on, we denote $\mathcal{V}_{C_{7}}$ and $\mathcal{V}_{F_{21}}$ for $C_{7}$-invariant and $F_{21}$-invariant polynomials respectively. 

By straightforward calculation we have
\begin{align*}
    \mathcal{V}_{C_{7}} = &\text{Span}_{\mathbb{C}}\{x_1^2x_2, x_3^2x_4, x_5^2x_6, x_2^2x_3, x_4^2x_5, x_6^2x_1, x_1x_3x_5,x_2x_4x_6\} \\
    \mathcal{V}_{F_{21}} = &\text{Span}_\mathbb{C}\{x_1^2x_2+x_3^2x_4+x_5^2x_6, x_2^2x_3+x_4^2x_5+x_6^2x_1, x_1x_3x_5,x_2x_4x_6\}
\end{align*}

Therefore, we define the following moduli spaces:
\begin{align*}
&\mathcal{F}_{C_{7}} = N_{\PGL(6,\CC)}(C_{7})\backslash\mathbb{P}\mathcal{V}_{C_{7}}^{sm} \\
&\mathcal{F}_{F_{21}} = N_{\PGL(6,\CC)}(F_{21})\backslash\mathbb{P}\mathcal{V}_{F_{21}}^{sm}
\end{align*}
and their GIT compactifications $\overline{\mathcal{F}}_{C_{7}}$ and $\overline{\mathcal{F}}_{F_{21}}$ respectively. 

Next we compute normalizer of $C_{7}$ and $F_{21}$ in $\PGL(6,\CC)$.

\begin{prop}
\label{NorC7}
    The normalizer $N_{\PGL(6,\mathbb{C})}(C_{7})$ of $C_{7}$ in $\PGL(6,\mathbb{C})$ as above is isomorphic to $(\mathbb{C}^{*})^{5}\rtimes C_{6}$. 

    Here $(\mathbb{C}^{*})^{5}$ is the subgroup of diagonal matrices in $\PGL(6,\CC)$. And $C_{6}$ is generated by the permutation matrix $(123456)$. 
\end{prop}

\begin{lem}
\label{cenC7}
    For $g\in \GL(6,\CC)$, denote its image in $\PGL(6,\CC)$ by $\overline{g}$. If $\overline{gg_{7}g^{-1}}=\overline{g_{7}}$, then $gg_{7}g^{-1}=g_{7}$. In this case, $g$ must be a diagonal matrix.
\end{lem}

\begin{proof}[Proof of Proposition \ref{NorC7}]
    We have $gg_{7}g^{-1}=\lambda g_{7}$ for some $\lambda\in\CC^{\times}$. Thus $\lambda g_{7}$ is similar to $g_{7}$ and then $\lambda \zeta_{7}=\zeta_{7}^{k}$ for some $k\in \ZZ$. Thus $\lambda=\zeta_{7}^{k-1}$. If $k\neq 1$, then $\lambda g_{7}$ has eigenvalue $1$, this is a contradiction since $g_{7}$ has no eigenvalue $1$. So we have $gg_{7}g^{-1}=g_{7}$. 

    We have $gg_{7}=g_{7}g$, hence $g$ must be diagonal.
\end{proof}

Now we can compute $N_{\PGL(6,\CC)}(C_{7})$.

\begin{proof}
    We consider the representation $\widetilde{\rho}$ of $C_{7}$ as above and consider the induced projective representation $\rho\colon C_{7}\rightarrow \PGL(6,\mathbb{C})$. Consider an arbitrary $\overline{g}\in \PGL(6,\CC)$ such that $\overline{g}$ normalizes $C_{7}$ (here we denote $C_{7}$ by its image in $\GL(6,\CC)$ and $\PGL(6,\CC)$ again if there is no ambiguity), then $\overline{gg_{7}g^{-1}}=\overline{g_{7}}^{k}$ for some $k\in \ZZ$.

    Let $g_{\tau}=(123456)\in \GL(6,\CC)$. Then $g_{\tau}$ stabilizes $C_{7}$ since $g_{\tau}g_{7}g_{\tau}^{-1}=g_{7}^{3}$. Thus one can check $\overline{gg_{7}g^{-1}}=\overline{g_{\tau}^{m}g_{7}g_{\tau}^{-m}}$ for some $m\in \ZZ/7\ZZ$ since $3$ generates $(\ZZ/7\ZZ)^{*}$. Thus after a left action of $\langle \overline{g_{\tau}}\rangle$, we may assume $\overline{gg_{7}g^{-1}}=\overline{g_{7}}$, that is $\overline{g_{7}}$ centralizes $\overline{g_{7}}$. Take a representative $g$ of $\overline{g}$ in $\GL(6,\CC)$, then $g$ is diagonal by Lemma \ref{cenC7}.

    In conclusion, for any element which normalizes $C_{7}$ in $\PGL(6,\CC)$, after a left multiplication of $\langle \overline{g_{\tau}}\rangle$, we may assume it is a centralizer of $C_{7}$. And it is easy to check that centralizer of $C_{7}$ is the group of diagonal matrices.  We denote $(\CC^{\times})^{5}$ by all diagonal matrices in $\PGL(6,\CC)$. Conversely, the group generated by $(\CC^{\times})^{5}$ and $\langle \overline{g_{\tau}}\rangle$ inside $\PGL(6,\CC)$ is $(\CC^{\times})^{5}\rtimes \langle \overline{g_{\tau}}\rangle \cong (\CC^{\times})^{5}\rtimes C_{6}$ and it normalized $C_{7}$. Thus the normalizer of $C_{7}$ in $\PGL(6,\CC)$ is $(\CC^{\times})^{5}\rtimes C_{6}$.
\end{proof}

\begin{lem}
\label{NF21inNC7}
 We have $N_{\PGL(6,\mathbb{C})}(F_{21})<N_{\PGL(6,\mathbb{C})}(C_{7})$, where $C_{7}<F_{21}$ is the unique Sylow-7 subgroup.   
\end{lem}

\begin{proof}
    Note that $F_{21}\cong C_{7}\rtimes C_{3}$, by Sylow's theorem, there is a unique order-7 subgroup, denoted by $C_{7}$. Any element in $\PGL(6,\CC)$ normalizing $F_{21}$ must sending $C_{7}$ into another order-7 subgroup of $F_{21}$, which is forced to be $C_{7}$ itself. Thus we have $N_{\PGL(6,\CC)}(F_{21})<N_{\PGL(6,\CC)}$, which is isomorphic to $ (\CC^{\times})^{5}\rtimes C_{6}$, see Lemma \ref{NorC7}.
\end{proof}

\begin{prop}
\label{NF21}
    The normalizer $N_{\PGL(6,\mathbb{C})}(F_{21})$ of $F_{21}$ in $\PGL(6,\mathbb{C})$ as above is isomorphic to $(\mathbb{C}^{\times}\times C_{21})\rtimes C_{6}$.

    Here $\mathbb{C}^{\times}$ is the subgroup consisting of elements of the form $\diag(1,c,1,c,1,c)$, $c\in \CC^{\times}$. The group $C_{21}$ is generated by $\frac{1}{7}(1,5,4,6,2,3)$ and $\frac{1}{3}(0,0,1,1,2,2)$. And $C_{6}$ is generated by the permutation matrix $(123456)$. 
\end{prop}

\begin{proof}
     
    Now we consider the linear representation and induced projective representation of $F_{21}$ as above.

    Recall that $g_{7}=\frac{1}{7}(1,5,4,6,2,3)$ and $g_{3}=(135)(246)$.

    Direct computation shows that $\overline{g_{\tau}}\in N_{\PGL(6,\CC)}(F_{21})$, where $g_{\tau}=(123456)$. Take $\overline{g}\in N_{\PGL(6,\CC)}(F_{21})$, then $\overline{g}\in N_{\PGL(6,\CC)}(C_{7})$. Since we have $g_{\tau}g_{7}g_{\tau}^{-1}=g_{7}^{3}$, after a left multiplication of $\langle \overline{g_{\tau}} \rangle$, we may assume $\overline{g}$ centralizes $C_{7}$. By Lemma \ref{cenC7}, $\overline{g}$ is diagonal. 
    
    After all discussions above, the condition that $\overline{g}$ normalizes $F_{21}$ is equivalent to that $\overline{g}$ conjugating $g_{3}$ to another element of order 3. All order-3 elements in $F_{21}$ can be written as $g_{7}^{m}g_{3}^{n}$ for some $m\in \ZZ/7\ZZ$ and $n=1\ \text{or}\ 2$. And $g_{7}^{m}g_{3}^{n}=g_{7}^{m*(1-2n)^{-1}}g_{3}^{n}g_{7}^{-m*(1-2n)^{-1}}$. So after further a left multiplication of $C_{7}$ (which does not change the property of centralizing $C_{7}$), we may further assume that $\overline{g}$ conjugates $\overline{g_{3}}$ to $\overline{g_{3}}$ or $\overline{g_{3}}^{2}$. The latter case can never happen since $\overline{g}$ is diagonal. So $\overline{g}$ centralizes $F_{21}$.    
    
    Take a representative $g$ of $\overline{g}$ in $\GL(6,\CC)$, then $g$ is diagonal and $gg_{3}g^{-1}=\lambda g_{3}$ with $\lambda \in \CC^{\times}$. Since $\lambda g_{3}$ is similar to $g_{3}$, $\lambda$ is $1$, $\omega$ or $\omega^{2}$, where $\omega=\text{exp}(\frac{2\pi\sqrt{-1}}{3})$. Let $k\coloneqq \frac{1}{3}(0,0,1,1,2,2)$. Note that $\omega^{n}g_{3}=k^{n}g_{3}k^{-n}$, $n=0,1,2$ and that $gg_{3}g^{-1}=g_{3}$ with $g$ diagonal is equivalent to that $g=\diag(a,b,a,b,a,b)$ for some $a,b\in \CC^{\times}$. We know that $gg_{3}g^{-1}=\omega^{n} g_{3}$ is equivalent to that $g=k^{n}\diag(a,b,a,b,a,b)$ for $n=0,1,2$ and $a,b\in\CC^{\times}$. 

    In conclusion, after a left multiplication of $\langle \overline{g_{\tau}} \rangle$, $C_{7}$ and $\langle\overline{k}\rangle$, we may assume $\overline{g}\in N_{\PGL(6,\CC)}(F_{21})$ is of the form $\overline{\diag (a,b,a,b,a,b)}=\overline{\diag(1,c,1,c,1,c)}$ with $n=0,1,2$ and $a,b,c=\frac{b}{a}\in\CC^{\times}$. So $N_{\PGL(6,\CC)}(F_{21})<\langle\CC^{\times},\overline{g_{7}},\overline{k},\overline{g_{\tau}}\rangle\cong (\CC^{\times}\times C_{21})\rtimes C_{6}$. Conversely, it is clear that $(\CC^{\times}\times C_{21})\rtimes C_{6}<N_{\PGL(6,\CC)}(F_{21})$. So the conclusion follows.
\end{proof}

\subsection{GIT Moduli of Cubic Fourfolds with $F_{21}$-Action}

\label{subsection: GIT moduli of F21}

By definition, there is a surjective morphism $\SL(6,\CC) \to \PGL(6,\CC)$ and the normalizer of $G_{1}=\mu_6C_7$ or $G_{2}=\mu_6F_{21}$ in $\SL(6,\CC)$ is just the inverse image of the normalizer of $C_7$ or $F_{21}$ in $\PGL(6,\CC)$. 

Using these, we analyze the GIT stability and form the following
\begin{prop}
\label{C7GITstb}
    For $F = a_1x_1^2x_2 + a_2x_2^2x_3 + a_3x_3^2x_4 + a_4x_4^2x_5 + a_5x_5^2x_6 + a_6x_6^2x_1 + a_7x_1x_3x_5 + a_8x_2x_4x_6$, it is $N_{\SL(6,\CC)}(G_1)$-semistable if and only if one of the following holds:
    \begin{center}
        (1) $a_7,a_8 \ne 0$ \,
        (2) $a_2,a_4,a_6,a_7 \ne 0$ \,
        (3) $a_1,a_3,a_5,a_8 \ne 0$ \,
        (4) $a_1,a_2,..,a_6 \ne 0$.
    \end{center}
    It is $N_{\SL(6,\CC)}(G_1)$-stable if and only if $a_1,a_2,..,a_6 \ne 0$. 
\end{prop}
\begin{proof}
    The connected component of $N_{\SL(6,\CC)}(G_1)$ containing the identity is just the subgroup of $\SL(6,\CC)$ containing all diagonal matrixes. Therefore, nontrivial $1$-parameter subgroups of $N_{\SL(6,\CC)}(G_1)$ are just $\lambda\colon \CC^\times \to N_{\SL(6,\CC)}(G_1)$ with $\lambda(t) = \diag(t^{r_1},\cdots,t^{r_6})$ such that $r_i \in \ZZ$, $\vec{r} \ne \vec{0}$ and $\sum\limits_{i=1}^6 r_i = 0$.

    Let $\omega_{a_1} = 2r_1+r_2,\cdots,\omega_{a_6}=2r_6 + r_1, \omega_{a_7}=r_1+r_3+r_5, \omega_{a_8} = r_2+r_4+r_6$ and let $\omega_{\vec{r}} = \min\{\omega_{a_i}|a_i \ne 0\}$. By Hilbert-Mumford criterion \cite[Theorem 2.1]{mumford1994geometric}, $F$ is not semi-stable (resp. stable) if and only if there is some $\vec{r} \ne \vec{0}$, such that $\omega_{\vec{r}} > 0$(resp. $\omega_{\vec{r}} \ge 0$).

    For sufficiency, first assume $a_7,a_8 \ne 0$, then there is no $\vec{r}$ such that $\omega_{a_7}, \omega_{a_8} > 0$ otherwise $\omega_{a_7} + \omega_{a_8} = \sum\limits_{i=1}^{6}r_i > 0$. The others are similar. For necessity, suppose none of four conditions hold. Up to symmetry, we may assume $a_1, a_7=0$, take $\vec{r} = (-25, -1,3 , 1, -1, 23)$, we would have $\min\{\omega_{a_i}|i=2,3,4,5,6,8\} > 0$ and therefore $F$ is not semistable. For stability, just assume $a_1=0$ and take $\vec{r} = (-8, -5, 10, 1, -2, 4)$ and repeat the same argument.    
\end{proof}

\begin{prop}
\label{F21GIT}
    For $F = c_1(x_1^2x_2+x_3^2x_4+x_5^2x_6) + c_2(x_2^2x_3+x_4^2x_5+x_6^2x_1) + c_3x_1x_3x_5 + c_4x_2x_4x_6$, it is $N_{\SL(6,\CC)}(G_2)$-stable if and only if one of the following holds:
    \begin{center}
     (1) $c_1,c_2 \ne 0$ \, (2) $c_1,c_4 \ne 0$ \, (3) $c_2,c_3 \ne 0$ \, (4) $c_3,c_4 \ne 0$.
   \end{center}
And $F$ is $N_{\SL(6,\CC)}(G_2)$-semistable implies it is $N_{\SL(6,\CC)}(G_2)$-stable.
In particular, the unstable locus is two projective lines $c_1 = c_3 = 0$ and $c_2 = c_4 = 0$. 
\end{prop}
\begin{proof}
    The connected component of $N_{\SL(6,\CC)}(G_2)$ containing the identity is just \[\mathbb{C}^\times = \diag(a,a^{-1},a,a^{-1},a,a^{-1}) < \SL(6,\CC).\]
    
    Therefore, all non-trivial $1$-parameter subgroups are just \[\lambda(t) = \diag(t^n,t^{-n},t^n,t^{-n},t^n,t^{-n}), n \in \ZZ/\{0\}.\] 
    
    Also by Hilbert-Mumford criterion, we obtain the conclusion.  
\end{proof}

\subsection{Singular Cubic Fourfolds with Order-$7$ Action}
\label{subsection: singular cubic fourfolds}

In this section, we study singular cubic fourfolds which admit an order-7 action. 

\begin{prop}
\label{nonADE}
For a cubic fourfold $X$ with an order-7 action, we write the defining equation of $X$ as $F = a_1x_1^2x_2 + a_2x_2^2x_3 + a_3x_3^2x_4 + a_4x_4^2x_5 + a_5x_5^2x_6 + a_6x_6^2x_1 + a_7x_1x_3x_5 + a_8x_2x_4x_6$. If one of $a_{i}=0$, $i=1,\dots, 6$, then $X$ admits singularity which is not of type ADE.
\end{prop}

\begin{proof}
    By \cite[Theorem 5.6]{laza2009moduli}, if $X$ has singularity at worst of type ADE, then $X$ is stable in the moduli of all cubic fourfolds. In particular, the stabilizer in $\PGL(6,\mathbb{C})$ of $X$ is finite and $\PGL(6,\mathbb{C})$-orbit of $X$ is closed in vector space of all cubic fourfolds. By Luna \cite[\S 3, Corollary 1]{luna1975adherences}, $N_{\PGL(6,\mathbb{C})}(C_{7})$-orbit of $X$ is closed. This forces $X$ to be stable in $\mathcal{V}_{C_7}$, which contradicts Proposition \ref{C7GITstb}.
\end{proof}

Direct computation shows that:
\begin{prop}
    For a cubic fourfold $X$ admitting symplectic $F_{21}$-action, write defining equation of $X$ as $F = c_1(x_1^2x_2+x_3^2x_4+x_5^2x_6) + c_2(x_2^2x_3+x_4^2x_5+x_6^2x_1) + c_3x_1x_3x_5 + c_4x_2x_4x_6$. If $c_{1}$ or $c_{2}$ equals $0$, then singular locus of $X$ has dimension at least $1$.
\end{prop}

Thus for the remaining cases, we may assume $X$ is defined by the equation  $x_{1}^{2}x_{2}+x_{3}^{2}x_{4}+x_{5}^{2}x_{6}+x_{2}^{2}x_{3}+x_{4}^{2}x_{5}+x_{6}^{2}x_{1}+ax_{1}x_{3}x_{5}+bx_{2}x_{4}x_{6}$ after suitable action of $N_{\PGL(6,\CC)}(C_{7})$.

\begin{lem}
\label{singeq}
    Suppose $F(x_1,\cdots,x_6) = x_1^2x_2+x_2^2x_3+\cdots+x_5^2x_6 + x_6^2x_1+ax_1x_3x_5 + bx_2x_4x_6$ has a singularity $[t_1:\cdots:t_6]$, then $t_i \ne 0$ for any $i$. Therefore, we can assume $t_1 = 1$ and set $t = t_2$. Then there exists some $0\le k \le2$ and $0\le l\le 6$ such that the following equations hold
\[
\left\{
	\begin{aligned}
		&a + 2\zeta_7^lt + \omega^k(\zeta_7^lt)^2 = 0\\
		&b + 2\omega^k(\zeta_7^lt)^{-1} + (\zeta_7^lt)^{-2} = 0
	\end{aligned}
\right.
\](here $\zeta_7 = e^{\frac{2\pi i}{7}}$ and $\omega = e^{\frac{2\pi i}{3}}$)

And $[1:t_2:\cdots :t_6] = [1:t:\omega^k\zeta_7^l:t\omega^k\zeta_7^{-2l}:\omega^{-k}\zeta_7^{-2l}:t\omega^{-k}\zeta_7^{-3l}]$. The converse also holds. 
\end{lem}
\begin{proof}
    A singularity $[x_1:x_2:\cdots : x_6]$ in $Z(F)$ satisfies the following equations: 
        \[\begin{cases}            x_6^2+2x_1x_2+ax_3x_5=0,\quad x_2^2+2x_3x_4+ax_5x_1=0,\quad x_4^2+2x_5x_6+ax_1x_3=0\\
            x_1^2+2x_2x_3+bx_4x_6=0,\quad x_3^2+2x_4x_5+bx_6x_2=0,\quad x_5^2+2x_6x_1+bx_2x_4=0
        \end{cases}\]
    If one of $x_i$ is zero, we shall prove that all $x_i$ are zero. Without loss of generality, we assume $x_1=0$. Then equations reduce to following:
    \begin{equation*}
        \begin{cases}
        x_6^2+ax_3x_5=0,&x_2^2+2x_3x_4=0,\quad \quad \quad \quad x_4^2+2x_5x_6=0
        \\2x_2x_3+bx_4x_6=0,&x_3^2+2x_4x_5+bx_6x_2=0,~~x_5^2+bx_2x_4=0
        \end{cases}
    \end{equation*}
    If any of $x_2,\cdots,x_6$ is zero, then all $x_2,\cdots,x_6$ would be forced to be zero. Thus from now on, we only consider cases where $x_2,\cdots,x_6$ are all nonzero.

    Then we actually just need to consider the 5 equations, from the second to the sixth. 
    They gives 
    \begin{equation*}
        \begin{cases}
            -b=\frac{2x_2x_3}{x_4x_6}=\frac{x_3^2+2x_4x_5}{x_6x_2}=\frac{x_5^2}{x_2x_4}\\
            x_2^2x_3+2x_3^2x_4=0\\
            x_4^2x_5+2x_5^2x_6=0
        \end{cases}
    \end{equation*}
    Then the first equation reduced to 
    \[2x_2^2x_3=x_3^2x_4+2x_4^2x_5=x_5^2x_6\]
    Let $e_i=x_i^2 x_{i+1}$ we have the following:
    \begin{equation*}
        \begin{cases}
            2e_2=e_3+2e_4=e_5\\
            e_2+2e_3=0\\
            e_4+2e_5=0\\
        \end{cases}
    \end{equation*}
    Plug the latter two equations into the first one we obtain that $-4e_3=e_3+2e_4=-\frac{1}{2}e_4$, we obtain $e_2,\cdots,e_5$ should be all zero.

    Now, if none of the $x_i$ is zero. Still denote $e_i=x_i^2x_{i+1}$, then multiply the six equations each with $x_1,x_2,x_3,x_4,x_5,x_6$ in order, the equations can be reduced to following:
    \begin{equation*}
        \begin{cases}
            e_6+2e_1+ax_1x_3x_5=0,
            \quad e_2+2e_3+ax_1x_3x_5=0,
            \quad e_4+2e_5+ax_1x_3x_5=0\\
            e_1+2e_2+bx_2x_4x_6=0,
            \quad e_3+2e_4+bx_2x_4x_6=0,
            \quad e_5+2e_6+bx_2x_4x_6=0
        \end{cases}
    \end{equation*}
    Hence we have 
    \begin{equation*}
        \begin{cases}
            e_6+2e_1=e_2+2e_3=e_4+2e_5\\
            e_1+2e_2=e_3+2e_4=e_5+2e_6
        \end{cases}
    \end{equation*}
    Solving these linear equations of $e_i$, we have all $e_1=e_3=e_5$, and $e_2=e_4=e_6$. Then it's not hard to find out that $[x_{1},x_{2},\dots,x_{6}]$ can be written as the form in the proposition. Plug the results into the equations, we have singularities exists if and only if there exists some $0\leq k\leq 2$ and $0\leq l\leq 6$ such that the following equations hold
    \begin{equation*}
        \begin{cases}
            a+2\zeta_7^l t+\omega^k (\zeta_7^{l} t)^2=0\\
        b+2\omega^k (\zeta_7^l t)^{-1}+(\zeta_7^l t)^{-2}=0
        \end{cases}
    \end{equation*}
\end{proof}

\begin{prop}
\label{singlocus}
    For a cubic fourfold $X=Z(F)$ with the following defining equation:
    
    \[x_{1}^{2}x_{2}+x_{3}^{2}x_{4}+x_{5}^{2}x_{6}+x_{2}^{2}x_{3}+x_{4}^{2}x_{5}+x_{6}^{2}x_{1}+ax_{1}x_{3}x_{5}+bx_{2}x_{4}x_{6}\]
    it is singular if and only if it lies in one of the following 3 curves inside $\CC^2$: $Z(G)$, $\omega Z(G)$ and $\omega^2 Z(G)$ where $G=a^2b^2-6ab+4a+4b-3$ and $\omega=e^{\frac{2\pi i}{3}}$. The cubic fourfold has exactly 7 isolated singularities if it lies on exactly one of the curves.
\end{prop}
\begin{proof}
    The equations about $a,b$ in the last proposition can be rewritten as follow. Replace $\zeta_7 t$ by $t$, we have singularities exists if and only if there exists $k$ and $t$ s.t. 
    \begin{equation}
        \begin{cases}
            a+2t+\omega^k t^2=0\\
        b+2\omega^k t^{-1}+ t^{-2}=0
        \end{cases}
    \end{equation}
    This gives 3 curves in $\mathbb{C}^2$.
    Let $C$ be the curve where $k=1$, then one sees that $\omega C$, $\omega^2 C$ give exactly another 2 curves.

    Then we only need to find the defining equations of $C$.

    For $(a,b)\in C$, there exists $t\in \mathbb{C}$ s.t. $a+2t+t^2=0$ and $b+2t^{-1}+t^{-2}=0$. Then $(t+1)^2=1-a$, and $\frac{(t+1)^2}{t^2}=1-b$. If one of $a,b$ is 1, then both $a,b$ is 1. When $a,b\neq 1$, then $t^2=\frac{1-a}{1-b}$. Then $a+2t+t^2=0$ implies $4t^2=(a+t^2)^2$, hence $4\frac{1-a}{1-b}=(a+\frac{1-a}{1-b})^2$. We have 
    \[4(1-a)(1-b)=(a-ab+1-a)^2\]
    This can be reduced to 
    \[a^2b^2-6ab+4a+4b-3=0\]
    Notice that $a,b=1$ is a solution of this equation.
    One can varifies for any $a,b$ satisfying this equation, either $a,b=1$, or $a,b\neq 1$. When $a,b=1$, then $t=-1$ satisfies our need, thus $(1,1)\in C$. When $a,b\neq 1$, let $t=-\frac{1}{2}\frac{1-ab}{1-b}$. Then $t^2=\frac{1}{4}\frac{(1-ab)^2}{(1-b)^2}=\frac{1-a}{1-b}$, thus $a+2t+t^2=a-\frac{1-ab}{1-b}+\frac{1-a}{1-b}=0$, and $b+2\frac{1}{t}+\frac{1}{t^2}=b-1+\frac{(1+t)^2}{t^2}=b-1+(1-a)\frac{1-b}{1-a}=0$. Hence the curve $C$ has defining equation $G=a^2b^2-6ab+4a+4b-3=0$.
\end{proof}

\begin{prop}
\label{Proposition:typeADEsinginabplane}
    For singular cubic fourfolds $X=Z(F)$ where $F$ is of the form 
    \[
F_{a,b}=x_{1}^{2}x_{2}+x_{3}^{2}x_{4}+x_{5}^{2}x_{6}+x_{2}^{2}x_{3}+x_{4}^{2}x_{5}+x_{6}^{2}x_{1}+ax_{1}x_{3}x_{5}+bx_{2}x_{4}x_{6},
\]
generically, they only admits $\mathrm{A}_1$ singularities. The only exceptions are $(a,b) = (1,1),(\omega,\omega)$ or $(\omega^2,\omega^2)$, and they lie in the same orbit. The cubic fourfold $Z(F_{1,1})$ admits exactly $7$ singularities of type $\mathrm{A}_2$. 
\end{prop}
\begin{proof}
    Suppose $F_{a,b}$ has a singularity, by Lemma \ref{singeq}, we may assume the singularity is $[1:t:\omega^k\zeta_7^l:t\omega^k\zeta_7^{-2l}:\omega^{-k}\zeta_7^{-2l}:t\omega^{-k}\zeta_7^{-3l}]$, where $t$ satisfies the equation in Lemma \ref{singeq}. For simplicity, let $c = \omega^k\zeta_7^l$, then the singularity is $[1:t:c:c^{19}t:c^5:c^{11}t]$. Meanwhile, $a = -2c^{15}t-c^{16}t^2, b = -2c^{13}t^{-1}-c^{12}t^{-2}$. 

    Consider the affine equation $F(1,x_2,..,x_6)$ and compute its Hessian matrix multiplied by $\frac{1}{2}$ at the point $(t,c,c^{19}t,c^5,c^{11}t)$, and compute its determinant, we obtain $-\frac{147}{16}c^{14}t(ct+1)$. Therefore, if it is a singularity not of type $\mathrm{A}_1$, then we would have $t = -c^{-1}$. Put it back into the expressions of $a,b$, we get three possibilities $(a,b) = (1,1),(\omega,\omega)$ or $(\omega^2,\omega^2)$. From the discussion above, these three cubic fourfolds lie in the same orbit and we only need to further analyze the singular locus of $F$ with $(a,b)=(1,1)$. 

    Note that when $(a,b) = (1,1)$, it has $7$ singularities of the form $[1:-\zeta_7^{6l}:\zeta_7^l:-\zeta_7^{4l}:\zeta_7^{5l}:-\zeta_7^{3l}]$ with $0\le l \le 6$. Since the action of $G_1 = C_7$ acts transitively on these $7$ singularities, we only need to consider $[1:-1:1:-1:1:-1]$. Consider the affine equation $F(1,x_2,..,x_6)$, and translate the point $(-1,1,-1,1,-1)$ to the origin, we obtain $f(x_2,..,x_6) = F(1,x_2-1,x_3+1,x_4-1,x_5+1,x_6-1)$. Take the matrix \[
    \widehat{S} = \begin{pmatrix}
1&0&0&0&0\\1&1&0&0&0\\3/4&1/4&1&0&0\\1&1/2&1&1&0\\1&0&1&0&1
    \end{pmatrix} 
    \]
    We obtain $\widehat{S}\cdot f = x_2^2-2x_3^2+\frac{7}{8}x_4^2-\frac{7}{4}x_5^2+x_6^3+\sum\limits_{i=2}^5x_ig_i(x_2,..,x_6)$, $g_i \in \mathfrak{m}^2$ where $\mathfrak{m}=(x_2,...,x_6)$ is a maximal ideal of $\CC[x_2,..,x_6]$. 

    We can further make coordinate changes and assume that $\widehat{S}\cdot f = x_2^2+x_3^2+x_4^2+x_5^2+x_6^3+\sum\limits_{i=2}^5x_ig_i(x_2,..,x_6)$. And then make coordinate changes $x_i \mapsto x_i-\frac{1}{2}g_i$ to obtain $x_2^2+x_3^2+x_4^2+x_5^2+x_6^3+a_4x_6^4+\sum\limits_{i=2}^5x_ih_i(x_2,..,x_6)$ with $h_i \in \mathfrak{m}^3$. Repeat in this way, we obtain $x_2^2+x_3^2+x_4^2+x_5^2+ux_6^3 + g$ where $u$ is a unit in $\CC[[x_6]]$ and $g \in \mathfrak{m}^n$ for $n$ sufficiently high. Using the finite determinacy theorem, see for example \cite{Greuel2007IntroductionTS}, we know that it has the same singularity type as $x_2^2+x_3^2+x_4^2+x_5^2+ux_6^3$ and therefore $x_2^2+x_3^2+x_4^2+x_5^2+x_6^3$, which is an $\mathrm{A}_2$ singularity. 
\end{proof}

Note that $N_{\PGL(6,\CC)}(F_{21})$-orbit of generic $F_{a,b}$ intersects the $(a,b)$-plane at $F_{a,b}$, $F_{b,a}$, $F_{\omega a,\omega b}$, $F_{\omega b,\omega a}$, $F_{\omega^2 a,\omega^2 b}$, $F_{\omega^2 b, \omega^2 a}$, thus in the finite quotient $C_6\bs\CC^2$, the 3 curves $Z(G),\omega Z(G), \omega^2 Z(G)$ becomes $Z(G)$ quotient by a single relation $(a,b)\sim (b,a)$, i.e. a curve isomorphic to $Z(x^2-6x+4y-3=0)$ where $x=ab, y=a+b$. We have the following proposition.

\begin{prop}
\label{Proposition:ADEminussmooth}
$\calF_{F_{21}}^{\textup{ADE}}-\calF_{F_{21}}\cong \Gamma_T\backslash(\calH_*-\calH_s)$ is an irreducible curve.   
\end{prop}

For a cubic fourfold $X$ with exactly one node and a $G$-action, one can relate $X$ to a degree-6 K3 surface $S$ with $G$-action. Conversely, any degree-6 K3 surface $S$ with $G$-action gives a one dimensional family of cubic fourfolds with $G$-action and the central fiber is a cubic fourfold with a node, see  \cite[\S 5.5, Page 1493, Page 1494]{laza2022automorphisms}.

In our case, for a cubic fourfold $X$ with an order-7 action admitting some nodes, one can similarly relates a surface by simulating the above construction. Specifically, we have the following conclusion:

\begin{prop} 
    For each cubic fourfold $X$ with an order-7 action and with 7 nodes as in Proposition \ref{singlocus}, around each node one can construct a degree 6 surface $S$ with at least 6 singular points and with a $C_{3}\cong F_{21}/C_{7}$-action.
    
    The $C_{3}$ action permutes these 6 singular points and they forms 2 $C_{3}$-orbits. And these 7 surfaces are isomorphic to each other with isomorphisms induced from the order-7 action.
\end{prop}
\begin{proof}
    For such a cubic fourfold $X = Z(F)$, we can make a coordinate change such that $F = x_6q(x_1,..,x_5) + c(x_1,..,x_5)$ with $q,c$ being degree $2$ and $3$ homogeneous polynomials respectively. Inside $\PP^5$, we can form a complete intersection $S = \{q = c = 0\}$ and this surface parametrizes all lines through the node $p = [0:0:0:0:0:1]$ of $Z(F)$. 

    It is a classical construction that, if $[0:\cdots:0:1]$ is the only singularity of $X$ and it is a node, then $\{q = c = 0\}$ would a K3 surface. However, here we have got six other nodes $p_1,..,p_6$. By Bézout's theorem, the line $\overline{p p_i}$ should lie in $X$, otherwise they should intersect $X$ with multiplicity $3$ but here $p,p_i$ are two nodes and therefore with multiplicity at least $4$. 

    Note that each of these lines $\overline{pp_i}$ corresponds to a singularity of $S$. Since $p_i$ is a singularity of $F$, we have $\frac{\partial F}{\partial x_j}(p_i) = x_6\frac{\partial q}{\partial x_j} + \frac{\partial c}{\partial x_j} = 0$ for $j = 1,..,5$. Therefore, the Jacobian matrix $\frac{\partial(q,c)}{\partial(x_1,..,x_5)}$ has rank $\le 1$ at these lines and therefore they are singularities. The rest of the proposition is a direct application of Lemma \ref{singeq}. 
\end{proof}

\begin{rmk}
As these surfaces are singular, the above proposition does not conflict with \cite[Theorem 5.4]{laza2022automorphisms}.
\end{rmk}

\subsection{GIT Moduli of Cubic Fourfolds with $L_2(7)$-Action}
\label{subsecL27}

In this section, we compute the moduli of cubic fourfold with $L_{2}(7)$-action.

It is a well-known fact that $L_{2}(7)$ can be realized as a subgroup of $S_{7}$ generated by permutations $(12)(36)$ and $(1234567)$. We denote them by $g_{2}$ and $g_{7}$ respectively.

Fix a $\CC^{7}$ with natural basis $\{e_{1},\dots,e_{7}\}$. We identify $\GL(6,\CC)$ with subgroup of $\GL(7, \CC)$ fixing the subspace $\text{Span}_{\CC}\{6e_1-e_2-\cdots-e_{7},6e_{2}-e_{1}-e_{3}-\cdots-e_{7},\dots,6e_{6}-e_{1}-\cdots-e_{5}-e_{7}\}$ and the element $e_{1}+\cdots+e_{7}$. Explicitly, there is a homomorphism $\Psi\colon\  \GL(6,\CC) \longrightarrow \ \ \GL(7,\CC)$
sending   $A$  to $\text{T}\begin{pmatrix}
        A & 0\\
        0 &1
    \end{pmatrix} T^{-1}$, where 
    \[T=\begin{pmatrix}
    6 & -1 & -1 & -1 & -1 & -1 & 1\\
    -1 & 6 & -1 & -1 & -1 & -1 & 1\\
    -1 & -1 & 6 & -1 & -1 & -1 & 1\\
    -1 & -1 & -1 & 6 & -1 & -1 & 1\\
    -1 & -1 & -1 & -1 & 6 & -1 & 1\\
    -1 & -1 & -1 & -1 & -1 & 6 & 1\\
    -1 & -1 & -1 & -1 & -1 & -1 & 1
    \end{pmatrix}\]
    
On the other hand, we consider the representation $\widetilde{\rho}:L_{2}(7)\rightarrow \GL(7,\CC)$ determined $g_{2}\mapsto (12)(36)\in \GL(7,\CC)$ and $g_{7}\mapsto (1234567)\in\GL(7,\CC)$.

We find that $\widetilde{\rho}$ factor through the subgroup $\Psi\colon \GL(6,\CC)\hookrightarrow\GL(7,\CC)$ by $\widetilde{\rho}=\Psi\circ \rho$. Here $\rho$ is the unique irreducible 6-dimensional representation of $L_{2}(7)$. We identify $g_{2}$ and $g_{7}$ with their image described as above.

\begin{prop}
\label{NL27}
    $N_{\PGL(6,\CC)}(L_{2}(7))=L_{2}(7)\rtimes C_{2}\cong \PGL(2,\FF_{7})$.
\end{prop}

\begin{proof}
    We have the following exact sequence:
\[1\rightarrow C_{\PGL(6,\CC)}(L_{2}(7))\rightarrow N_{\PGL(6,\CC)}(L_{2}(7))\rightarrow \Aut(L_{2}(7)),\]
where $C_{\PGL(6,\CC)}(L_{2}(7))$ is the centralizer and $\Aut(L_{2}(7))$ is the automorphism group of $L_{2}(7)$.

We claim that $C_{\PGL(6,\CC)}(L_{2}(7))$ is trivial and the rightmost arrow is surjective. So, $N_{\PGL(6,\CC)}(L_{2}(7))\cong \Aut(L_{2}(7))\cong L_{2}(7)\rtimes C_{2}\cong \PGL(2,\FF_{7})$, see \cite[\S3.4] {steinberg1960automorphisms}.

Let $\overline{g}\in C_{\PGL(6,\CC)}(L_{2}(7))$, take a lifting $g\in \GL(6,\CC)$. Then $gg_{2}g^{-1}=\lambda g_{2}$ and $gg_{7}g^{-1}=\mu g_{7}$ for some $\lambda,\mu \in \CC^{\times}$. By comparing eigenvalues of $\lambda g_{2}$, $\mu g_{7}$ with $g_{2}$ and $g_{7}$ respectively, we know $\lambda=\mu=1$ and $g\in C_{\GL(6,\CC)}(L_{2}(7))$. So $g$ is an automorphism of $\CC^{6}$ as an irreducible $L_{2}(7)$-module, see Lemma \ref{repinv}. By Schur's lemma, $g$ is a scalar and thus $\overline{g}$ is trivial.

For surjectivity, consider 
\[E=(243756)*(a(\Id+ g_{7}^{4}+g_{7}^{6})+b(g_{7}+g_{7}^{2}+g_{7}^{3}+g_{7}^{5})),\]
where $a=\frac{2\sqrt{2}}{7}+\frac{1}{7}$ and $b=\frac{-3\sqrt{2}}{14}+\frac{1}{7}$. Direct computation shows that $E\in \GL(6,\CC)<\GL(7,\CC)$ described above, $E^{2}=(162)(457)\in L_{2}(7)$ and $\overline{E}\in$ $N_{\PGL(6,\CC)}(L_{2}(7))$. And $\overline{E}\notin L_{2}(7)$ since $E$ is not a permutation matrix. 
\end{proof}

From \cite[\S 7.2]{KOIKE202512}, we know that cubic fourfolds with faithful $L_2(7)$-action are given as
\[
\left\{
	\begin{aligned}
		&x_1+x_2+x_3+x_4+x_5+x_6+x_7 = 0, \\
		&c_1(x_1^3+x_2^3+x_3^3+x_4^3+x_5^3+x_6^3+x_7^3)\\&+c_2(x_1x_2x_4+x_2x_3x_5+x_3x_4x_6+x_1x_5x_6 + x_1x_3x_7 + x_4x_5x_7 + x_2x_6x_7) = 0,
	\end{aligned}
\right.
\]
and they are invariant under the action of the representation $\widetilde{\rho}\colon L_2(7) \to \GL(7,\CC)$ above. 

Let $F_1 = x_1^3+x_2^3+x_3^3+x_4^3+x_5^3+x_6^3+x_7^3$ and $F_2 = x_1x_2x_4+x_2x_3x_5+x_3x_4x_6+x_1x_5x_6 + x_1x_3x_7 + x_4x_5x_7 + x_2x_6x_7$, then $f_1 = F_1((x_1,...,x_6,0)7*T^{-1})$ and $f_2 = F_2((x_1,...,x_6,0)7*T^{-1})$ are invariant polynomials over $\CC^6$ under the action of the representation $\rho\colon L_2(7) \to \SL(6,\CC)$ above. Note that
\[
7*T^{-1} = 
\begin{pmatrix}
    1&0&0&0&0&0&-1\\
    0&1&0&0&0&0&-1\\
    0&0&1&0&0&0&-1\\
    0&0&0&1&0&0&-1\\
    0&0&0&0&1&0&-1\\
    0&0&0&0&0&1&-1\\
    1&1&1&1&1&1&1\\
\end{pmatrix}. 
\]
Therefore, $f_1$ and $f_2$ are obtained from replacing $x_7$ by $-(x_1+\cdots+x_6)$ in $F_1$ and $F_2$ respectively. 

From \cite[Theorem 1.2]{laza2022automorphisms} and \cite[Proposition 4.10]{yu2020moduli}, we know that the family of cubic fourfolds with $L_2(7)$-action form a one-dimensional family. Since $N_{\PGL(6,\CC)}(L_2(7))$ is $0$-dimensional and $f_1,f_2$ are clearly linearly independent, we know that the space $\mathcal{V}_{L_2(7)}$ is just the subspace generated by $f_1,f_2$. 

Now let $E^\prime = \Psi^{-1}(E)$, we have $E^{\prime}(f_1) = \frac{3\sqrt{2}}{2}f_2$, $E^{\prime}(f_2) = \frac{\sqrt{2}}{3}f_1$, this makes the action of $N_{\PGL(6,\CC)}(L_2(7))/L_2(7) \cong C_2$ on $\PP(\mathcal{V}_{L_2(7)})$ clear. Therefore, we have the following proposition. 

\begin{prop}
    Denote $\mathcal{F}_{L_2(7)}$ by $N_{\PGL(6,\CC)}(L_{2}(7))\backslash\mathbb{P}\mathcal{V}_{L_{2}(7)}^{sm}$, and $\overline{\mathcal{F}}_{L_2(7)}$ by its GIT compactification, then we have $\overline{\mathcal{F}}_{L_2(7)}\cong  C_{2}\backslash\mathbb{P}^{1}\cong\mathbb{P}^{1}$.
\end{prop}

By intersecting the family $\mathbb{P}\mathcal{V}_{L_{2}(7)}\cong \mathbb{P}^{1}$ with the singular locus described in Proposition \ref{singlocus} in $\mathbb{P}\mathcal{F}_{G_{1}}$, we can characterize the singular cubic fourfolds with $L_{2}(7)$-action.

\begin{prop}
\label{singuL27}
    Inside the $\mathbb{P}\mathcal{V}_{L_{2}(7)}\cong \mathbb{P}^{1}$, there are exactly 6 points corresponding to 6 singular cubic fourfolds with $L_{2}(7)$-action. Two of them have $7$ nodes, another two have $14$ nodes, and remaining two cubic fourfolds have singular locus as a smooth quadratic surface, and they are known as determinantal cubic fourfolds. Each pair of cubic fourfolds as above contributes 1 point in $\overline{\mathcal{F}}_{3}$. The equations for those six singular cubic fourfolds are given in Table \ref{tab:sing}. 
\end{prop}
\begin{proof}
    Since we have already made clear the singular cubic fourfolds in the $F_{21}$-family, we make some coordinate change to let the $\mathcal{V}_{L_2(7)}$ lie in $\mathcal{V}_{F_{21}}$. 

    Explicitly, we first find some copy of $F_{21}$ inside $\widetilde{\rho}(L_2(7))$, which is generated by $g_7^{\prime\prime} = (1234567)$ and $g_3^{\prime\prime} = (235)(476)$. Also, let $g_7^{\prime} = \Psi^{-1}(g_7^{\prime \prime})$ and $g_3^{\prime} = \Psi^{-1}(g_3^{\prime \prime})$. To conjugate this $F_{21}$ generated by $g_7^{\prime}$ and $g_3^{\prime}$ to the one we are already familiar with, which is the one generated by $g_3\colon x_i \mapsto x_{i+2}$ and $g_7 = \frac{1}{7}(1,5,4,6,2,3)$, we need to find some isomorphism of representation. 
    
    Here we use a trick and let $S = \sum\limits_{i=0}^2\sum\limits_{j=0}^{6}(g_3^ig_7^j)^{-1}(g_3^{\prime})^i(g_7^{\prime})^j$. Therefore, we would have $Sg_7^\prime S^{-1} = g_7$ and $Sg_3^\prime S^{-1} = g_3$. The invariant polynomials with respect to $S\rho(L_2(7))S^{-1}$ would be linear combinations of $S\cdot f_1$ and $S \cdot f_2$, generating an $\PP^1$ inside $\PP(\mathcal{V}_{F_{21}})$. 

    Recall that cubic fourfolds with non-isolated singular locus in $\PP(\mathcal{V}_{F_{21}})$ give two projective planes, and the $\PP^1$ above intersects them in exactly $2$ points. It is easy to check that they are in the same orbit as the determinantal cubic fourfolds. 

    Meanwhile, it can be checked that $\frac{1}{21}S\cdot f_1 + \frac{t}{7}S \cdot f_2$ and $[1:1:2-\frac{7ct}{(t+c)^2}:2+\frac{7(\frac{1}{2}+c)t}{(t-\frac{1}{2}-c)^2}]$ in $\PP(\mathcal{V}_{F_{21}})$ are in the same orbit, where $c = -\frac{1}{2}\zeta_7^4-\frac{1}{2}\zeta_7^2-\frac{1}{2}\zeta_7-\frac{1}{2}$ (For example, using the action of $\diag(1,s,1,s,1,s)$). Using Proposition \ref{singlocus}, we obtain explicit equations of singular cubic fourfolds in the $L_2(7)$-family, see Table \ref{tab:sing}. Note that two equations in each row can be permuted to each other using the action of $E^{\prime}$ above. 
\end{proof}

\begin{longtable}{ |p{3cm} p{3cm}||p{6cm}| }
 \hline
 Equations& &Singular Type\\
 \hline
  $f_1 + 3\frac{1+\sqrt{-7}}{4}f_2$& $f_1 + 3\frac{1-\sqrt{-7}}{4}f_2$& Determinantal cubic fourfolds \\
  \hline
  $f_1 - \frac{3}{10}f_2$& $f_1 - 15f_2$& 7 nodes\\
  \hline
  $f_1 - \frac{3}{2}f_2$& $f_1 - 3f_2$& 14 nodes\\
 \hline
  \caption{Singular cubic fourfolds inside $\PP(\mathcal{V}_{L_2(7)})$}
  \label{tab:sing}
\end{longtable}

\section{Full Automorphism Groups of Cubic Fourfolds Admitting Order-$7$ Action}
\label{section:full automorphism group}

In this section we determine all possible full automorphism groups of cubic fourfolds admitting an order-7 action. Precisely, we prove Theorem \ref{theorem: main}.

For any smooth cubic fourfold $X$, the symplectic automorphism group $\Aut^{s}(X)$ is a normal subgroup of the automorphism group $\Aut(X)$. Meanwhile, there is a natural embedding $\Aut(X) \to \PGL(6,\CC)$ according to \cite{Matsumura1963OnTA}. So any element of $\Aut(X)$ lies in the normalizer of $\Aut^s(X)$ in $\PGL(6, \CC)$. Therefore we have the following lemma:

\begin{lem}
\label{AutAuts}
For $X\in\PP\mathcal{V}_{G}$ a smooth cubic fourfold with $\Aut^{s}(X)=G$, we have $\Aut(X)=\mathrm{Stab}_{N_{\PGL(6,\CC)}(G)}(X)$. Here $\mathrm{Stab}_{N_{\PGL(6,\CC)}(G)}(X)$ is the stabilizer of $X$ under the action of $N_{\PGL(6,\CC)}(G)$.
\end{lem}

\subsection{Case $F_{21}$}
For $F_{21}$ case, recall that $\PP\mathcal{V}_{F_{21}}\cong \PP^{3}$. We will use $[a,b,c,d]\in \PP^{3}$ to represent a cubic fourfold $X\in \PP\mathcal{V}_{F_{21}}$ with defining equation $a(x_{1}^{2}x_{2}+x_{3}^{2}x_{4}+x_{5}^{2}x_{6})+b(x_{2}^{2}x_{3}+x_{4}^{2}x_{5}+x_{6}^{2}x_{1})+cx_{1}x_{3}x_{5}+dx_{2}x_{4}x_{6}$.

    The action of $N_{\PGL(6,\CC)}(F_{21})$ on $\PP\mathcal{V}_{F_{21}}$ with respect to the coordinates described as above determines a projective representation $R\colon N_{\PGL(6,\CC)}(F_{21}) \rightarrow \PGL(4,\CC)$. 

    Directly compute all possible stabilizers of $X\in \PP\mathcal{V}_{F_{21}}$ for the action of $N_{\PGL(6,\CC)}(F_{21})$. Orbit-stabilizer theorem shows that orbits which contain some smooth cubic fourfolds are disjoint union of some affine curves. More precisely, we will obtain the following list for such orbits:
\begin{longtable}{ |p{6.1cm}||p{3cm}|p{3cm}|p{1.2cm}|  }
 \hline
 cubic fourfold & stabilizer &number of orbit components & orbit degree\\
 \hline
 $[1,a,b,c]$, $a,b,c\neq 0$, $c\neq a^{3}b$ & $F_{21}$ & 6 & 3\\
 \hline 
 $[1,a,0,b]$ or $[1,a,b,0]$, $a,b,c\neq 0$ & $F_{21}$ & 6 & 2\\
 \hline
 $[1,a,b,a^{3}b]$, $a,b\neq 0$ & $F_{21}\rtimes C_{2}$ & 3 & 3\\
 \hline
 $[1,a,0,0]$, $a\neq 0$ & $F_{21}\rtimes C_{6}$ & 1 & 1\\
 \hline
 \caption{Stabilizers and Orbits}
\end{longtable}

For smooth cubic fourfold $X=[1,a,b,c]$ with $a\neq 0$ and $c\neq a^{3}b$, generically $X$ is smooth with $\Aut^{s}(X)=F_{21}$. In this case, $\Aut(X)=F_{21}$ by Lemma \ref{AutAuts}.

For cubic fourfold $X=[1,a,0,0]$, $X$ is isomorphic to the Klein fourfold $[1,1,0,0]$. We have $X$ is smooth and $\Aut^{s}(X)=F_{21}$. So $\Aut(X)=F_{21}\rtimes C_{6}$.

We claim that for a generic $X=[1,a,b,a^{3}b]$ with $a,b\neq 0$, we have $X$ is smooth and $\Aut^{s}(X)=F_{21}$. Thus $\Aut(X)=F_{21}\rtimes C_{2}$ once $\Aut^{s}(X)=F_{21}$.

Consider the sub-family $\mathbb{A}^{2}\cong\{[1,a,b,a^{3}b]\}\subset \PP\mathcal{V}_{F_{21}}$. Note that $[1,1,0,0]\in \mathbb{A}^{2}$ and $X=[1,1,0,0]$ is smooth and $\Aut^{s}(X)=F_{21}$. The claim follows since smooth cubic fourfolds $X\in \PP\mathcal{V}_{F_{21}}$ with $\Aut^{s}(X)=F_{21}$ forms a dense subset containing an open subset of $\PP\mathcal{V}_{F_{21}}$.

Combining Lemma \ref{repinv}, we have actually taken into consideration all smooth cubic fourfold $X$ with $\Aut^{s}(X)=F_{21}$.

\subsection{Case $L_2(7)$ and $A_7$}
For $G=L_{2}(7)$, recall that $\PP\mathcal{V}_{L_{2}(7)}\cong \PP^{1}$.  We use $[a,b]$ to represent a cubic fourfold $X\in\PP\mathcal{V}_{L_{2}(7)}$ with defining equation $af_{1}+bf_{2}$, see \S \ref{subsecL27} for expression of $f_{1}$ and $f_{2}$.

Note that $N_{\PGL(6,\CC)}(L_{2}(7))=L_{2}(7)\rtimes C_{2}$ acts on $\PP\mathcal{V}_{L_{2}(7)}$ factor through the quotient $C_{2}$. And $C_{2}$ acts on $\PP\mathcal{V}_{L_{2}(7)}$ by sending $[a,b]$ to $[\frac{3\sqrt{2}}{2}b,\frac{\sqrt{2}}{3}a]$. 

So for $X\neq [1,\frac{3}{2}\sqrt{2}],[1,-\frac{3}{2}\sqrt{2}]$ with $X$ smooth and $\Aut^{s}(X)=L_{2}(7)$, we have $\Aut(X)=L_{2}(7)$.

For $X=[1,\frac{3}{2}\sqrt{2}] \ \text{or}\ [1,-\frac{3}{2}\sqrt{2}]$, by proposition $\ref{singuL27}$, we have $X$ is smooth. We know $\Aut^{s}(X)\neq A_{7}$. Otherwise, by \cite[Theorem 1.8 (2)]{laza2022automorphisms}, $L_{2}(7)\rtimes C_{2}<A_{7} \ \text{or} \ S_{7}$. However, $N_{S_{7}}(L_{2}(7))=N_{A_{7}}(L_{2}(7))=L_{2}(7)$. So $\Aut^{s}(X)=L_{2}(7)$. By Lemma \ref{AutAuts} again, $\Aut(X)=L_{2}(7)\rtimes C_{2}$.

Denote them by $X_{1}$ and $X_{2}$ respectively. Above argument shows that they are Galois conjugate to each other with $\text{Gal}(\QQ(\sqrt{2}/\QQ))$. The following proposition shows that they are not isomorphic to each other.

\begin{prop}
\label{notiso}
Let $X_{1}\neq X_{2}\in \PP \mathcal{V}_{L_{2}(7)}$ with automorphism group $L_{2}(7)\rtimes C_{2}$. Then $X_{1}$ is not isomorphic to $X_{2}$. 
\end{prop}

\begin{proof}
    We first prove $N_{\PGL(6,\CC)}(L_{2}(7)\rtimes C_{2})=N_{\PGL(6,\CC)}(L_{2}(7))$. Take $\overline{g}\in N_{\PGL(6,\CC)}(L_{2}(7)\rtimes C_{2})$. Denote $\overline{g}L_{2}(7)\overline{g}^{-1}$ by $G$. Then $G$ is a normal subgroup of $L_{2}(7)\rtimes C_{2}$ and thus $G\cap L_{2}(7)$ is a normal subgroup of $L_{2}(7)$. Since $L_{2}(7)$ is a simple group, $G\cap L_{2}(7)$ must be $L_{2}(7)$ itself or $\{ \id\}$. On the other hand, $G/G\cap L_{2}(7)=1\ \text{or}\ C_{2}$, which forces $G=L_{2}(7)$.

    Recall that $\Aut(X_{1})$ and $\Aut(X_{2})$ are exactly the same subgroup of $\PGL(6,\CC)$, isomorphic to $L_{2}\rtimes C_{2}$. If $X_{1}\cong X_{2}$, then one can find $\overline{g}\in \PGL(6,\CC)$ such that $\overline{g}X_{1}=X_{2}$ by \cite{Matsumura1963OnTA}. Then $\overline{g}\Aut(X_{1})\overline{g}^{-1}<\Aut(X_{2})$. So $\overline{g}\in N_{\PGL(6,\CC)}(L_{2}(7)\rtimes C_{2})=L_{2}(7)\rtimes C_{2}$. We know that $\overline{g}\in L_{2}(7)\rtimes C_{2}=\Aut(X_{1})$. So $X_{1}=X_{2}$, this is a contradiction.
\end{proof}

\begin{rmk}
    When embedding $X_{1}$ and $X_{2}$ into $\PP\mathcal{V}_{F_{21}}$, they are of the type $[1,a,b,a^{3}b]$ since there is an extra order-2 element in $N_{\PGL(6,\CC)}(L_{2}(7))=\Aut(X_{i})$ normalizing $F_{21}$, $i=1,2$.

    Act them to the plane $\mathbb{A}^{2}=\{[1,1,a,b]\}\subset \PP\mathcal{V}_{F_{21}}$, then they are in the diagonal.
\end{rmk}
    
\begin{rmk}
    We have $E^{\prime}\cdot F_{X_1}=F_{X_1}$ and $E^{\prime}\cdot F_{X_2}=-F_{X_2}$ where $F_{X_i}$ is the defining equation of $X_{i}$, $i=1,2$. 

    So $X_{1}$ and $X_{2}$ are corresponds to different symmetric type for cubic fourfolds with $L_{2}(7)\rtimes C_{2}$-action, see \cite[\S 2.2]{yu2020moduli}.
\end{rmk}

For $G=A_{7}$, \cite[Theorem 1.8 (2)]{laza2022automorphisms} shows that there are 2 smooth cubic fourfolds with symplectic automorphism group equals $A_{7}$. By \cite[Theorem 1.8 (2)]{laza2022automorphisms}, $X$ defined by $f_{1}$ is smooth, satisfying $\Aut^{s}(X)=A_{7}$ and $\Aut(X)=S_{7}$. And \cite[Theorem 6.14]{yang2024automorphism}, \cite[Page 16 (2.3)]{KOIKE202512} give explicit expressions of the other smooth $X$ with $\Aut^{s}(X)=\Aut(X)=A_{7}$.

In conclusion, we have proved Theorem \ref{theorem: main}.

\section{More Geometric Consequences}
In this section, we apply the previous explicit analysis to obtain more geometric consequences.
\subsection{Moduli Spaces $\mathcal{F}_{C_7}$ and $\mathcal{F}_{F_{21}}$}
We first show that $\overline{\mathcal{F}}_{F_{21}}$ is a rational surface. From the proof, we can also determine $\mathcal{F}_{F_{21}}^{\mathrm{ADE}}$.
\begin{prop}
\label{modulirati}
    There is a $C_6$ action on $\mathbb{A}^2$ which gives an open embedding $C_6\backslash\mathbb{A}^2 \hookrightarrow \overline{\mathcal{F}}_{F_{21}}$, which identifies $C_{6}\backslash\mathbb{A}^2$ with $\mathcal{F}_{F_{21}}^{\mathrm{ADE}}$.
    
    Moreover, $C_{6}\backslash\mathbb{A}^2$ is rational and therefore $\overline{\mathcal{F}}_{F_{21}}$ is rational. 
\end{prop}

\begin{proof} 
    Consider $\mathbb{A}^{2}=\{[1,1,a,b]| \ a,b\in \CC\}\subset\PP\mathcal{V}_{F_{21}}$. Consider that $C_{6}$ acts on $\mathbb{A}^{2}=\{[1,1,a,b]|\ a,b\in \CC\ \}$ by sending $[1,1,a,b]$ to $[1,1,\omega b,\omega a]$. Computing the intersection of $\mathbb{A}^2$ with $N_{\PGL(6,\CC)}(F_{21})$-orbits inside $\PP(\mathcal{V}_{F_{21}})$, we can show that the natural embedding $\mathbb{A}^{2}\rightarrow \PP\mathcal{V}_{F_{21}}$ induces an injective morphism

    \[
    C_{6}\backslash\mathbb{A}^{2}\rightarrow \overline{\mathcal{F}}_{F_{21}}\]
    which maps onto an open subvariety. Since $\overline{\mathcal{F}}_{F_{21}}$ is normal, so is this open subvariety. Therefore, applying Zariski's main theorem, we know that $C_{6}\backslash\mathbb{A}^2$ is isomorphic to its image and therefore $C_{6}\backslash\mathbb{A}^2 \hookrightarrow \overline{\mathcal{F}}_{F_{21}}$ is an open embedding. 

    The property that $\mathcal{F}_{F_{21}}^{\mathrm{ADE}}\cong C_{6}\backslash\mathbb{A}^{2}$ is a direct result from Proposition \ref{nonADE} and Proposition \ref{Proposition:typeADEsinginabplane}. 
    
    Therefore we only need to show that $C_{6}\backslash\mathbb{A}^{2}$ is rational. By computing invariants, we have 
    \[
    C_{6}\backslash\mathbb{A}^{2}\cong\text{Spec}\,\CC[(\alpha+\beta)^{3}, (\alpha\beta)^{3},\alpha\beta(\alpha+\beta)]\cong\text{Spec}\,\CC[x,y,z]/(xy-z^{3}).
    \] 
    It is easy to show that this affine cubic surface is rational by considering projection to the first and third coordinates. 
\end{proof}

Moreover, we have the following proposition, showing that it is equivalent to study $\overline{\mathcal{F}}_{C_{7}}$ and $\overline{\mathcal{F}}_{F_{21}}$.

\begin{prop}
\label{moduliiso}
    We have $\overline{\mathcal{F}}_{C_{7}} \cong \overline{\mathcal{F}}_{F_{21}}$ as algebraic varieties.
\end{prop}
\label{iso}
\begin{proof} 
Using Theorem \ref{yz}, we know that there is a natural finite morphism $\overline{\mathcal{F}}_{F_{21}} \to \overline{\mathcal{M}}$ which is in fact a normalization of its image since generically, cubic fourfolds with $F_{21}$-action has exactly automorphism group $F_{21}$. 

Meanwhile, there is a natural morphism $\PP(\mathcal{V}_{F_{21}}) \to \PP(\mathcal{V}_{C_{7}})$ sending $[a:b:c:d]$ to $[a:a:a:b:b:b:c:d]$ with respect to the basis given at the beginning. This morphism restricts to the semistable locus and gives $\PP(\mathcal{V}_{F_{21}}^{ss}) \to \PP(\mathcal{V}_{C_{7}}^{ss}) \to N_{\SL(6,\CC)}(C_{7})\dbs \PP(\mathcal{V}_{C_{7}}^{ss})$. Note that since $C_7$ is the only sylow-$7$ subgroup of $F_{21}$, we have $N_{\SL(6,\CC)}(F_{21}) < N_{\SL(6,\CC)}(C_{7})$. Therefore, the above morphism is equivariant under the action of $N_{\SL(6,\CC)}(F_{21})$ and descends to a morphism $\overline{\mathcal{F}}_{F_{21}} \to \overline{\mathcal{F}}_{C_{7}}$. This morphism forces the natural morphism $\overline{\mathcal{F}}_{C_{7}} \to \overline{\mathcal{M}}$ to be generically $1$ to $1$ and hence a normalization of its image.

On the other hand, it is easy to check that each element in $\PP(\mathcal{V}_{C_{7}})$ can become of the form $[a:a:a:b:b:b:c:d]$ after some action from $N_{\SL(6,\CC)}(C_{7})$. Therefore, $\overline{\mathcal{F}}_{C_{7}} \to \overline{\mathcal{M}}$ is mapped into the image of $\overline{\mathcal{F}}_{F_{21}} \to \overline{\mathcal{M}}$. Using the universial property of normalization, we obtain a natural morphism $\overline{\mathcal{F}}_{C_{7}} \to \overline{\mathcal{F}}_{F_{21}}$. 

The two morphisms $\overline{\mathcal{F}}_{C_{7}} \to \overline{\mathcal{F}}_{F_{21}}$ and $\overline{\mathcal{F}}_{F_{21}} \to \overline{\mathcal{F}}_{C_{7}}$ are clearly birational morphisms.  It is also clear that they are inverse to each other and we are done. 
\end{proof}

\subsection{Relation to Degree-Two K3}
\label{subsection: relation to degree-2 K3}

As \cite[Theorem 1.2 (8)]{laza2022automorphisms} mentioned, there is a relation between cubic fourfold with $L_{2}(7)$ action with degree-2 K3-surface. We give a more precise description in this section.

By proposition \ref{singuL27}, there exist exactly 2 determinantal cubic fourfolds in family $\mathbb{P}\mathcal{V}_{L_{2}(7)}$. We denote them by $X_{1}$ and $X_{2}$ respectively. The singularity locus of each of $X_{i}$ with $i=1,2$ is a quadratic smooth rational surface $V_{i}$ with $i=1,2$, which is given by Veronese embedding from $\mathbb{P}^{2}$ to $\mathbb{P}^{5}$. And $V_{i}$ naturally inherent the $L_{2}(7)-$action from $X_{i}$, $i=1,2$. For a generic cubic fourfold $X$ with a $G$-action with $C_{7}< G < L_{2}(7)$, the intersection of $X$ with each of $V_{i}$ gives a sextic curve $C_{i,X}$ admitting a $G$-action. 

There is a sextic curve $C$ with $G<\PGL(3,\CC)-$action is related to a degree-2 K3-surface $S$ with a $G$-action. Here $S$ is constructed by considering degree-2 covering of $\mathbb{P}^{2}$ branched over $C$. Explicitly, suppose $C$ has the defining equation $f(z_{1},z_{2},z_{3})$, a homogeneous polynomial with degree 6. Then $S\subset \PP(1,1,1,3)$ (the weighted projective space) is defined by the equation $y^{2}=f(z_{1},z_{2},z_{3})$, with $G$ acting on the coordinates $z_{1},z_{2},z_{3}$ the same as $G$ acts on $\PP^{2}$ and $G$ acting on $y$ trivially.

For $G=L_{2}(7)$, there is explicit expression of sextic curve with $L_{2}(7)$-action in \cite[Page 16]{smith2007picard} and \cite[Proposition 8.1]{badr2025stratification}. We provide an alternative approach with more details as following:

\begin{prop}   
    For any cubic fourfold $X\in \mathbb{P}\mathcal{V}_{L_{2}(7)}$ with $X\neq X_{2}$, $X\cap V_{1}$ is the same smooth sextic curve $C_{1}$. Thus the above construction relates the whole family $\mathbb{P}\mathcal{V}_{L_{2}(7)}\setminus \{X_{2}\}$ to one degree-2 K3 surface $S$ with an $L_{2}(7)$-action. Symmetrically, we have the same result for $X_{2}$ and $V_{2}$.
\end{prop}

\begin{proof}
    We have singular locus $V_{1}$ of $X_{1}$ is $L_{2}(7)-$invariant. So $X\cap V_{1}$ has a $L_{2}(7)-$action for $X\in\PP\mathcal{V}_{L_{2}(7)}$. Combining argument above, we only need to check $X\cap V_{1}$ is a smooth sextic curve.

    We have a determinantal cubic fourfold $X_{1}^{\prime}=[1,0,-2,-1]\in \PP\mathcal{V}_{F_{21}}$. The singular locus of $X_{1}^{\prime}$ is a surface $V_{1}^{\prime}=\{[z_{2}z_{3},z_{1}^{2},z_{1}z_{3},z_{2}^{2},z_{1}z_{3},z_{3}^{2}]|\ [z_{1},z_{2},z_{3}]\in \PP^{1}\ \}\cong \PP^{2}$. Let $s=-2\zeta_{7}^{5}-2\zeta_{7}^{3}-\zeta_{7}-1$, where $\zeta_{7}=\text{exp}(\frac{2\pi\sqrt{-1}}{7})$. There is \[g=\diag(1,s^{-1},1,s^{-1},1,s^{-1})\in N_{\PGL(6,\CC)}(F_{21})\] such that $gX_{1}=X_{1}^{\prime}$. Then singular locus $V_{1}$ of $X_{1}$ is $g^{-1}V_{1}^{\prime}$. For any $X\in \PP\mathcal{V}_{L_{2}(7)}$, $X\cap V_{1}=g^{-1}(gX\cap V_{1}^{\prime})$. So we only need to compute $gX\cap V_{1}^{\prime}$.

    For $X$ defined by $F(x_{1},\dots ,x_{6})$, $X\cap V_{1}^{\prime}\subset \PP^{2}$ is defined by $F(z_{2}z_{3},z_{1}^{2},z_{1}z_{3},z_{2}^{2},z_{1}z_{3},z_{3}^{2})$. 
    
    Direct computation shows that, up to scalar in $\CC^{\times}$, \[gSf_{1}(z_{2}z_{3},z_{1}^{2},z_{1}z_{3},z_{2}^{2},z_{1}z_{3},z_{3}^{2})=gSf_{2}(z_{2}z_{3},z_{1}^{2},z_{1}z_{3},z_{2}^{2},z_{1}z_{3},z_{3}^{2})=z_{1}^{5}z_{3}+z_{2}^{5}z_{1}+z_{3}^{5}z_{2}-5z_{1}^{2}z_{2}^{2}z_{3}^{2}.\]

    So $gSf_{1}\cap V_{1}^{\prime}=gSf_{2}\cap V_{1}^{\prime}=\{[z_{1},z_{2},z_{3}]|\ z_{1}^{5}z_{3}+z_{2}^{5}z_{1}+z_{3}^{5}z_{2}-5z_{1}^{2}z_{2}^{2}z_{3}^{2}=0\}$, which is a smooth sextic curve in $\PP^{2}$. For $X\in \PP\mathcal{V}_{L_{2}(7)}$ defined by $aSf_{1}+bSf_{2}$ for some $a,b\in \CC$, $gX\cap V_{1}^{\prime}$ is defined by $(agSf_{1}+bgSf_{2})(z_{2}z_{3},z_{1}^{2},z_{1}z_{3},z_{2}^{2},z_{1}z_{3},z_{3}^{2})$, giving either the same sextic curve as above or the whole $\PP^{2}$. The latter case can happen only when $gX=X_{1}^{\prime}$, which is equivalent to $X=X_{1}$. 
\end{proof}

\subsection{Boundary of $\overline{\mathcal{F}}_{F_{21}}$}
\label{subsection: boundaryofFF21}

Now we study the boundary of the moduli space $\overline{\mathcal{F}}_{F_{21}}\setminus \mathcal{F}_{F_{21}}$. We have obtained the following stratification:
\begin{align*}
    \overline{\mathcal{F}}_{F_{21}}\supset \mathcal{F}_{F_{21}}^{\mathrm{ADE}}=\mathbb{A}^{2}/C_{6}\supset \mathcal{F}_{F_{21}}.
\end{align*}
Recall that $\mathcal{F}_{F_{21}}^{\mathrm{ADE}}-\mathcal{F}_{F_{21}}$ is an irreducible curve described in Proposition \ref{Proposition:ADEminussmooth}. So, it remains to determine $\overline{\mathcal{F}}_{F_{21}}- \mathcal{F}_{F_{21}}^{\mathrm{ADE}}$.

By Proposition \ref{F21GIT} and Proposition \ref{modulirati}, we know that a cubic fourfold in $\overline{\mathcal{F}}_{F_{21}}- \mathcal{F}_{F_{21}}^{\mathrm{ADE}}$ can be written as the following types: $[1,0,a,b]$, $[0,1,b,a]$ and $[0,0,1,b]$ with $b\neq 0$.

Up to an $N_{\PGL(6,\CC)}(F_{21})$-action, the remaining types are $[1,0,t,-1]$ and $[0,0,1,1]$ with $t\in \mathbb{C}$. Moreover, two different cubic fourfolds $[1,0,t,-1]$ and $[1,0,t^{\prime},-1]$ are in the same $N_{\PGL(6,\CC)}(F_{21})$-orbit if and only if $t=-t^{\prime}$. 

So $\overline{\mathcal{F}}_{F_{21}}- \mathcal{F}_{F_{21}}^{\mathrm{ADE}}$ consists of two parts: image of $C_{2}\backslash\mathbb{A}^{1}$ induced by the natural injection $t\mapsto [1,0,t,-1]$, which is a rational irreducible curve, and a point represented by $[0,0,1,1]$. 

In conclusion, we obtain the following description of the boundary of $\overline{\mathcal{F}}_{F_{21}}$.

\begin{prop}
\label{Proposition:boundaryoverlineF21minusF21}
    The boundary $\overline{\mathcal{F}}_{F_{21}}- \mathcal{F}_{F_{21}}$ of moduli space of cubic fourfold with $F_{21}$-action consists of the following $3$ parts: $2$ irreducible curves and $1$ point.

    One of the irreducible curves consists of singular cubic fourfolds with singularities of ADE type and the other irreducible curve consists of cubic fourfolds of type $[1,0,t,-1]$. The point is given by the cubic fourfold $[0,0,1,1]$.
\end{prop}

From \cite{yu2020moduli}, we have the following commutative diagram with horizontal morphisms being isomorphisms. The complements of vertical left embeddings were described in Proposition \ref{Proposition:boundaryoverlineF21minusF21}.
\[\begin{tikzcd}
	{\mathcal{F}_{F_{21}}} && {\Gamma_{T}\backslash(\mathbb{D}-\mathcal{H}_{s})} \\
	\\
	{C_{6}\backslash\mathbb{A}^{2}\cong\mathcal{F}_{F_{21}}^{ADE}} && { \Gamma_{T}\backslash(\mathbb{D}- \mathcal{H}_{*})} \\
	\\
	{\overline{\mathcal{F}}_{F_{21}}} && {\Gamma_{T}\backslash\overline{\mathbb{D}}^{\mathcal{H}_{*}}}
	\arrow["\cong", from=1-1, to=1-3]
	\arrow[hook, from=1-1, to=3-1]
	\arrow[hook, from=1-3, to=3-3]
	\arrow["\cong", from=3-1, to=3-3]
	\arrow[hook, from=3-1, to=5-1]
	\arrow[hook, from=3-3, to=5-3]
	\arrow["\cong", from=5-1, to=5-3]
\end{tikzcd}\]

Since horizontal maps are isomorphisms, we also determine the boundary of quotient of Type IV domain $(\overline{
\Gamma_{T}\backslash\mathbb{D}}^{\mathcal{H}_{*}})- (\Gamma_{T}\backslash(\mathbb{D}-\mathcal{H}_{s})) $.

Note that the determinantal cubic fourfold $[1,0,-2,1]$ lies in the boundary of the moduli spaces $\overline{\mathcal{F}}_{F_{21}}$, so there is a corresponding point $p\in\overline{\Gamma_{T}\backslash\mathbb{D}}^{\mathcal{H}_{*}}$. This implies $\calH_*$ is nonempty. So the Looijenga compactification $\overline{\Gamma_{T}\backslash\mathbb{D}}^{\mathcal{H}_{*}}$ is different to the Baily-Borel compactification $\overline{\Gamma_{T}\bs\mathbb{D}}^{\mathrm{BB}}$.



\section{Periods of the $L_2(7)$-family}
\label{section: periods of L27}

In this section, we focus on the $L_2(7)$-family and determine its monodromy and period domain. In the end, we characterize the quotient curve $\Gamma_{T_1}\backslash\DD_{T_1}$ in an explicit way. For convenience, we use notation $T$ instead of $T_1$, $\Gamma$ instead of $\Gamma_{T_1}$ and $\Phi$ instead of $\Phi_{1}$ through this section. The calculation is inspired by \cite{dolgachev1996mirror}.

\subsection{Calculation of the Lattice $T$} 
\label{conslatL27}
We compute the lattice corresponding to the local period domain of $L_{2}(7)$-family. We mainly use Conway-Sloane's notation \cite[Chapter 15]{conway2013sphere} and the classification result of H\"{o}hn--Mason \cite[\S 4]{hohn2016290}. The reader can also refer to \cite[Appendix A]{laza2022automorphisms} for a brief introduction to the Conway-Sloane symbol.
\begin{defn}
A pair $(G, S)$ is called a Leech pair, if $S$ is a positive definite even lattice with no vectors of norm two, and $G$ is a finite group acting faithfully on $S$ such that the induced action of $G$ on the discriminant group $A_S$ is trivial, and the invariant sublattice $S^G$ is trivial.
\end{defn}

By \cite[Appendix B.2]{gaberdiel6symmetries} (see also \cite[Corollary 4.19]{marquand2025finite} and \cite[Lemma 1.1]{zheng2025lemmaleechlikelattices}), there is the following useful lemma: 
\begin{lem}
\label{lemma: leech pair}
Suppose $(G,S)$ is a Leech pair, and $\rank(S)+l(A_S)\le 24$. Then there exists a primitive embedding of $S$ into the Leech lattice $\LL$.
\end{lem}

Now let $G=L_2(7)$, and $S$ the coinvariant lattice of the action of $G$ on $\Lambda_0$. By \cite[Lemma 4.2]{laza2022automorphisms}, $(G,S)$ is a Leech pair satisfying the condition in Lemma \ref{lemma: leech pair}. Then there is a primitive embedding $S\hookrightarrow \LL$. Thus the $L_{2}(7)$-action on $S$ can be extended to $\mathbb{L}$, such that $S$ is the coinvariant lattice in $\mathbb{L}$. It is known from \cite{laza2022automorphisms} that the invariant sublattice $\mathbb{L}^{G}$ corresponds to \cite[Table 1, item 77]{hohn2016290}, which has discriminant form $q_{S}=4_{1}^{+1}7^{+2}$. And $S$ is the orthogonal complement of $\mathbb{L}^{G}$ in $\LL$, so $S$ has discriminant form $q_{S}=4_{7}^{+1}7^{+2}$. The lattice $T$ is the orthogonal complement of $S$ in $\Lambda_{0}$, and thus $T$ has signature $(1,2)$. We have the following embeddings:
\[
S\oplus E_{6} \hookrightarrow \Lambda_{0}\oplus E_{6}\hookrightarrow \text{\Rmnum{2}}_{26,2}. \]
The lattice $\text{\Rmnum{2}}_{26,2}$ is the even unimodular lattice of signature $(26,2)$, known as the Borcherds lattice. The lattice $T$ is the orthogonal complement of $S\oplus E_{6}$ in $\text{\Rmnum{2}}_{26,2}$. Then $q_T=4_{1}^{+1}7^{+2}3^{-1}$. By \cite[Theorem 1.13.2]{Nikulin_1980}, such $T$ is unique. One can directly construct a Gram matrix for $T$ with respect to certain basis as below:
\begin{equation}
    \label{equation: gram matrix of T}
\begin{pmatrix}
-2 & 1 & 0 \\ 1 & 10 & 0 \\ 0 & 0 & -28
\end{pmatrix}.
\end{equation}

\subsection{Characterization of $\Gamma$}
\label{subsection: chara of Gamma L27}
By abuse of notation, we denote the above matrix by $T$. 

According to \cite[\S 4.1.2]{laza2022automorphisms}, $(L_2(7), S)$ is a saturated Leech pair and therefore $L_2(7) = \{h \in \widehat{\Gamma}|h|_T = \Id_T\}$. Now let $N$ denote the normalizer of $L_2(7)$ inside $\widehat{\Gamma}$. For each element $f \in N$ and $h \in L_2(7)$, $f^{-1}hf = h^{\prime}$ for some $h^{\prime} \in L_2(7)$. Therefore, for any $t \in T$, we have $h(f(t)) = f(t)$, i.e. $f(t)$ is $L_2(7)$-invariant. By definition, we have $f(t) \in T$. Therefore, we can restrict $f$ on $T$. When forming quotients, there is no difference between considering $N$ or its image in $\rO^+(T)$ via restriction. Therefore, we denote $\Gamma = \mathrm{Im}(N \to \rO^+(T))$. We have the following characterization of $\Gamma$. 
\begin{prop}
    Recall that the discriminant group of $T$ is isomorphic to $\ZZ/4\ZZ \oplus (\ZZ/7\ZZ)^2\oplus\ZZ/3\ZZ$. An element $f \in \rO^+(T)$ is in $\Gamma$ if and only if $\overline{f}|_{\ZZ/3\ZZ}$ is trivial. 
\end{prop}
\label{gam}
\begin{proof}
From \cite[\S 4]{hohn2016290}, the discriminant group $A_S$ of $S$ is isomorphic to $\ZZ/4\ZZ \oplus (\ZZ/7\ZZ)^2$. For lattice $\Lambda_0$, we have $A_{\Lambda_0}\cong \ZZ/3\ZZ$. Let $A_T$ denote the discriminant group of $T$. From their discriminant groups we see that $\Lambda_0/(S\oplus T)$ is isomorphic to $\ZZ/4\ZZ \oplus (\ZZ/7\ZZ)^2$. Let $H_S = A_S$ and $H_T$ be the subgroup $\ZZ/4\ZZ \oplus (\ZZ/7\ZZ)^2$ of $A_T$. From \cite[Proposition 1.5.1]{Nikulin_1980}, we see that the primitive embedding $T \hookrightarrow \Lambda_0$ with orthogonal complement $S$ is determined by an isomorphism $\gamma\colon H_T \overset{\sim}{\to} H_S$ with $q_S \circ\gamma = -q_T$ (where $q_S$ and $q_T$ are their corresponding discriminant-quadratic forms). Also note that the projection $\Lambda_0^*/(S\oplus T) \to T^*/T=  A_T$ induces an embedding $A_{\Lambda_0} \to A_T$ to the copy of $\ZZ/3\ZZ$. 

    Now suppose $f \in \rO^+(T)$ satisfies that $\overline{f}|_{\ZZ/3\ZZ}$ is trivial. Note that from \cite[\S 4]{hohn2016290} we have $\rO(S) \to A_S = H_S$ being surjective, thus we can choose some $g \in \rO(S)$ such that $\overline{g} = \gamma \circ\overline{f} \circ \gamma^{-1}$. Therefore, also from \cite[Corollary 1.5.2]{Nikulin_1980}, we can extend $f$ to some $\widetilde{f} \in \rO^+(\Lambda_0)$ such that $\widetilde{f}|_T = f, \widetilde{f}|_S = g$. Note that $\widetilde{f} \in \widehat{\Gamma}$ since $\overline{f}|_{\ZZ/3\ZZ}$ is trivial. Meanwhile, $\widetilde{f}L_2(7)\widetilde{f}^{-1} = L_2(7)$ is easy to check, therefore, $f \in \Gamma$.

    Conversely, any $f \in \Gamma$ is the image of some element in $\widehat{\Gamma}$, then $f|_{\ZZ/3\ZZ}$ is trivial. The proposition follows. 
\end{proof}

For each element $f \in \rO^+(T)$, if $\overline{f}|_{\ZZ/3\ZZ}$ is not trivial, then it sends $1$ to $-1$ in $\ZZ/3\ZZ$, thus $-f$ acts trivially on $\ZZ/3\ZZ$. Meanwhile, it is obvious that $-\Id \notin \Gamma$, therefore we have the following
\begin{prop}
    As a group, $\rO^+(T)$ is generated by $\Gamma$ and $-\Id$. And $\Gamma$ is an index $2$ subgroup of $\rO^+(T)$. 
\end{prop}

Let $\rO(1,2)$ be the orthogonal group of $T \otimes \mathbb{R}$, and identify $\rO^+(T) < \rO(1,2)$. Multiplying those elements in $\Gamma$ with determinant $-1$ by $-\Id$, we obtain an isomorphism from $\Gamma$ to $\Gamma^\prime = \rO^+(T) \cap \SO(1,2)$. The quotients of $\mathbb{D}_T$ by $\Gamma$ or $\Gamma^\prime$ are obviously the same, therefore we are reduced to considering $\Gamma^\prime = \rO^+(T) \cap \SO(1,2)$. 

\subsection{Tube Domain Model for the $L_2(7)$-family}
\label{tubedomainL27}
Let $e_1,e_2,e_3 \in T$ be the integral basis for which the Gram matrix is \eqref{equation: gram matrix of T}. Then $(e_1)^2=-2,(e_2)^2=10,(e_1,e_2)=1,(e_3)^2=-28$. Let $f,e,g \in T\otimes\mathbb{R}$ be an $\RR$-basis such that if we denote
\[
B = \begin{pmatrix}
    -\frac{\sqrt{21}}{42}-\frac{1}{2} & 0 & -\frac{\sqrt{21}}{42}+\frac{1}{2} \\
    -\frac{\sqrt{21}}{21} & 0 & -\frac{\sqrt{21}}{21} \\
    0 & \frac{\sqrt{14}}{14} & 0
\end{pmatrix},
\]
we have $\begin{pmatrix}
    f&e&g
\end{pmatrix} = \begin{pmatrix}
    e_1&e_2&e_3
\end{pmatrix}B.$ 

Note that
\[
B^TTB = \begin{pmatrix}
    0&0&1\\
    0&-2&0\\
    1&0&0
\end{pmatrix}
\]

Therefore, $f,g$ are isotropic and $(f,g) = 1, (e)^2=-2, (f,e)=0, (g,e) = 0$. 

With this basis, we can identify $T \otimes \RR$ with the quadratic space of binary forms over $\RR$, sending $\alpha f + \beta e + \gamma g$ to $\alpha x^2 + 2\beta xy + \gamma y^2$, where the quadratic form on this space is $Q(\alpha,\beta,\gamma) = -2\beta^2+2\alpha \gamma$. Note that $\beta^2-\alpha \gamma$ is the discriminant of the binary form $\alpha x^2 + 2\beta xy + \gamma y^2$. Each $h = \begin{pmatrix}
    a & b \\ c & d
\end{pmatrix} \in \SL(2, \RR)$ can be seen as an isometry on this space via the action
\[
h \cdot [\begin{pmatrix}
    x & y
\end{pmatrix} 
\begin{pmatrix}
    \alpha & \beta \\
    \beta & \gamma
\end{pmatrix}  
\begin{pmatrix}
    x \\ y
\end{pmatrix}]
 = 
\begin{pmatrix}
    x & y
\end{pmatrix} 
\begin{pmatrix}
    a & b \\ 
    c & d
\end{pmatrix}
\begin{pmatrix}
    \alpha & \beta \\
    \beta & \gamma
\end{pmatrix}  
\begin{pmatrix}
    a & b \\ 
    c & d
\end{pmatrix}^T
\begin{pmatrix}
    x \\ y
\end{pmatrix}
\]

With respect to the basis $f,e,g$, the action of $h$ corresponds to a matrix $A(h)\in \SO^+(1,2) = \SO^+(1,2;\RR)$. Straightforward calculation shows that 
\[A(h) = \begin{pmatrix}
    a^2 & 2ab & b^2 \\ ac & ad+bc & bd \\ c^2 & 2cd & d^2
\end{pmatrix}.\]
Let us denote $\Phi\colon \SL(2,\RR) \to \SO^+(1,2)$ be the morphism given by sending $h$ to $BA(h)B^{-1}$ (i.e. considering the basis $e_1,e_2,e_3$). This gives a surjective morphism $\SL(2,\RR) \to \SO^+(1,2)$ of degree two, which induces a group isomorphism $\PSL(2, \RR) \overset{\sim }{\to} \SO^+(1,2)$. 

Let $\mathbb{H} = \{z \in \CC|\Im(z)>0\}$ be the complex upper plane, with $\SL(2, \RR)$ acting on it by $h\cdot \frac{\omega_1}{\omega_2} = \frac{a\omega_1+b\omega_2}{c\omega_1+d\omega_2}$. it is easy to check that there is a biholomorphic morphism $\mathbb{H} \to \mathbb{D}_T$ sending $z$ to $z^2f + ze + g$, or equivalently sending $\frac{\omega_1}{\omega_2}$ to $\omega_1^2f + \omega_1\omega_2e+\omega_2^2g$, which is the classic tube domain model. As binary forms, we can write $\omega_1^2f + \omega_1\omega_2e+\omega_2^2g$ as $\begin{pmatrix}
    x & y
\end{pmatrix} 
\begin{pmatrix}
    \omega_1 \\ \omega_2
\end{pmatrix} 
\begin{pmatrix}
    \omega_1 & \omega_2
\end{pmatrix}
\begin{pmatrix}
    x \\ y
\end{pmatrix}.$ 
Therefore, we have the following equivariant diagram
\[ \begin{tikzcd}
\SL(2, \RR) \arrow{r}{\Phi} \arrow[symbol=\circlearrowright]{d} & \SO^+(1,2) \arrow[symbol=\circlearrowright]{d} \arrow[hook]{r} & \mathrm{O}(1,2) \arrow[symbol=\circlearrowright]{d} \\
\mathbb{H} \arrow{r}{\cong} & \mathbb{D}_T \arrow[hook]{r} & \mathbb{P}(T \otimes \CC)
\end{tikzcd}
\]

Next we characterize the curve $\mathbb{D}_T/\Gamma$ as an arithmetic quotient of $\mathbb{H}$. For this, we compute the inverse image $\Phi^{-1}(\Gamma^{\prime}) = \Phi^{-1}(\rO^+(T)\cap \SO^+(1,2))$. By definition, they are those $h \in \SL(2, \RR)$ such that $\Phi(h)$ has integer coefficients. 

\begin{prop}
\label{proposition:arithgroupofL27}
    Denote $\Gamma^{\prime \prime} = \Phi^{-1}(\rO^+(T)\cap \SO^+(1,2))$. We have 
\begin{align*}
    \Gamma^{\prime \prime} = &\{
    \begin{pmatrix}
        \frac{u\sqrt{u^\prime}+v\sqrt{v^\prime}}{2} & \frac{w\sqrt{w^\prime}+x\sqrt{x^\prime}}{2} \\
        \frac{w\sqrt{w^\prime}-x\sqrt{x^\prime}}{2} &
        \frac{u\sqrt{u^\prime}-v\sqrt{v^\prime}}{2}
    \end{pmatrix}|u,v,w,x \in \ZZ, (u^\prime, v^\prime, w^\prime, x^\prime) = (1,21,6,14)\  \text{or} \ (2,42,3,7)\\ &\text{up to permutation by }\mathrm{V} < \mathrm{S}_4, \text{and }u^2u^\prime - v^2v^\prime - w^2w^\prime + x^2x^\prime = 4\}
\end{align*}
(Here we use $\mathrm{V}=\{1, (12)(34), (13)(24),(14)(23)\}$ to denote the Klein $4$-group).
\end{prop}
\begin{proof}
    For any $h=\begin{pmatrix}
        a & b\\ c&d
    \end{pmatrix} \in \SL(2, \RR)$ such that $\Phi(h)$ has integer coefficients. Set $\Phi(h) = N = \begin{pmatrix}
        n_1 & n_2 & n_3 \\ n_4 & n_5 & n_6\\ n_7 & n_8 & n_9
    \end{pmatrix}$ with $n_i \in \ZZ$. Writing out $A(h) = B^{-1}NB$, we obtain this matrix
\begin{frame}
\footnotesize
\setlength{\arraycolsep}{2.5pt} 
\medmuskip = 1mu 
\[ 
\left( \begin{array}{llllll} 
\frac{21(n_1+n_5)+(n_1+2n_2+10n_4-n_5)\sqrt{21}}{42}& -\frac{\sqrt{14}}{14}n_3 + (-\frac{\sqrt{6}}{4}+\frac{\sqrt{14}}{28})n_6 & \frac{21(-n_1+n_4+n_5)+(n_1+2n_2-11n_4-n_5)\sqrt{21}}{42}  \\
(-\frac{\sqrt{6}}{6}-\frac{\sqrt{14}}{2})n_7 - \frac{\sqrt{6}}{3}n_8 & n_9 & (-\frac{\sqrt{6}}{6}+\frac{\sqrt{14}}{2})n_7 - \frac{\sqrt{6}}{3}n_8 \\
\frac{21(-n_1+n_4+n_5)-(n_1+2n_2-11n_4-n_5)\sqrt{21}}{42} & \frac{\sqrt{14}}{14}n_3 + (-\frac{\sqrt{6}}{4}-\frac{\sqrt{14}}{28})n_6 & \frac{21(n_1+n_5)-(n_1+2n_2+10n_4-n_5)\sqrt{21}}{42}    
\end{array} \right)
\]
\end{frame}

Now let $u = a+d, v=a-d, w=b+c, x=b-c$, direct computation shows that $u^2, v^2, w^2, x^2 \in \ZZ$ and  $\sqrt{21}uv, \sqrt{21}wx, 2\sqrt{6}uw \in \ZZ$ (and therefore $\sqrt{21}uv, \sqrt{21}wx, \sqrt{6}uw \in \ZZ$). Therefore, we can set $u=u_1\sqrt{u_2},..,x=x_1\sqrt{x_2}$  with $u_1,..,x_1 \in \ZZ$, $u_2,..,x_2 \in \ZZ^+$ and $u_2,..,x_2$ square-free. Since $\sqrt{21}uv \in \ZZ$, we have $u_2v_2 = 21y_1^2$ for some $y_1 \in \ZZ^+$, $y_1$ square-free and cannot be divided by $3$ or $7$. This gives four possibilities, which are $u_2 = 21y_1$, $7y_1$, $3y_1$ or $y_1$. Similarily for $w_2,x_2$, we have $w_2x_2 = 21y_2^2, u_2w_2 = 6y_3^2$ with $y_2, y_3 \in \ZZ^+$, square-free, and not divided by $3,7$ and $2,3$ respectively. 

Up to symmetry, we only need to consider $u_2 = 6y_3 = 6w_2$ or $u_2 = 3y_3, w_2=2y_3$. In the first case, we must have $w_2$ not divided by $2,3$ and thus $w_2 = y_2$ or $7y_2$. Further discussion of the relation of $u_2$ and $v_2$ shows that $(u_2,v_2,w_2,x_2) = (6y_2,14y_2,y_2,21y_2)$ or $(42y_2,2y_2,7y_2,3y_2)$. The "determinant $= 1$" condition gives $u_1^2u_2-v_1^2v_2-w_1^2w_2+x_1^2x_2=4$ and therefore $4$ is divided by $y_2$. However, $v_2 = 14y_2$ or $2y_2$ shows that $y_2$ is not divided by $2$. The only possibility is that $y_2 = 1$. 

Similarily, in the second case, $u_2 = 3y_3, w_2=2y_3$ we will have $(u_2,v_2,$ $w_2,x_2) =$ $(21y_1,y_1,$ $14y_1,$ $6y_1)$ or $(3y_1,7y_1,$$2y_1,42y_1)$ and we must have $y_1 = 1$. Therefore, $(u_2,v_2,w_2,x_2)=$ $(1,$ $21,$ $6,14)$ or $(2,42,3,7)$ up to permutation by $\mathrm{V} < \mathrm{S}_4$. 

Conversely, for any element $h$ in the group in the statement of the theorem, we already have $\Phi(h)$ having integer coefficients, which can be seen applying elementary number theory. For example, when $h = \begin{pmatrix}
        \frac{u+v\sqrt{21}}{2} & \frac{w\sqrt{6}+x\sqrt{14}}{2} \\
        \frac{w\sqrt{6}-x\sqrt{14}}{2} &
        \frac{u-v\sqrt{21}}{2}
    \end{pmatrix}$, we have $u^2 - 21v^2 - 6w^2 + 14x^2 = 4$. Observing the coefficients of $\Phi(h)$, we see that $\Phi(h)$ has integer coefficients if and only if $u-v,w-x \in 2\ZZ$. (This is the same as $u$ is even if and only $v$ is even, $w$ is even if and only $x$ is even). They all can be deduced from $u^2 - 21v^2 - 6w^2 + 14x^2 = 4$. For example, if we modulo $2$, we obtain $u^2-v^2 \equiv 0$ and thus $u-v \in 2\ZZ$ or $u+v \in 2\ZZ$. However, $u-v\in 2\ZZ$ is the same as $u+v \in 2\ZZ$. Similar argument applies to $w,x$ if we modulo $4$. 
Therefore, the theorem is proven. 
\end{proof}

We investigate more information about the group $\Gamma^{\prime \prime}$ as below. For convenience, we call elements of the form $\begin{pmatrix}
        \frac{u\sqrt{u^\prime}+v\sqrt{v^\prime}}{2} & \frac{w\sqrt{w^\prime}+x\sqrt{x^\prime}}{2} \\
        \frac{w\sqrt{w^\prime}-x\sqrt{x^\prime}}{2} &
        \frac{u\sqrt{u^\prime}-v\sqrt{v^\prime}}{2}
    \end{pmatrix}$ in $\Gamma^{\prime \prime}$ have type $(u^\prime,v^\prime,w^\prime,x^\prime)$. There are eight types of elements here. 
\begin{prop}
\label{normal}
    Elements of type $(1,21,6,14)$ form a normal subgroup $H$ of index $8$ in $\Gamma^{\prime\prime}$, with each type of elements form a coset in $\Gamma^{\prime\prime}$ with repect to this subgroup $H$. Moreover, $\Gamma^{\prime \prime}/H \cong (\ZZ/2\ZZ)^{\oplus3}$. 
\end{prop}

\begin{proof}
    Direct computation shows that $H$ is a subgroup. For elements of type $(u^\prime,v^\prime,w^\prime,x^\prime)$, we can choose some element $r$ of this type (for example, see below). Also by direct computation, we can show that left multiplication by $r$ and $r^{-1}$ will give a bijection between elements of type $(1,21,6,14)$ and elements of type $(u^\prime,v^\prime,w^\prime,x^\prime)$. Therefore, the set of elements of type $(u^\prime,v^\prime,w^\prime,x^\prime)$ is just the left coset $rH$. Right multiplication will give the same conclusion, thus $H$ is a normal subgroup. 
\end{proof}

The following table will make the structure clear.
\begin{longtable}{ |p{2cm}||p{4cm}|p{3cm}|p{3cm}|  }
\hline
Type& Representative & Preserves $\ZZ/3\ZZ$ in $A_T$? & As elements in $(\ZZ/2\ZZ)^{\oplus3}$\\
\hline
$(1,21,6,14)$ & $\Id$ & Yes & $(0,0,0)$\\
\hline
$(21,1,14,6)$ & $\begin{pmatrix}
        \frac{\sqrt{21}+3}{2} & \frac{\sqrt{14}+\sqrt{6}}{2} \\
        \frac{\sqrt{14}-\sqrt{6}}{2} &
        \frac{\sqrt{21}-3}{2}
    \end{pmatrix}$ & No & $(1,0,0)$\\
\hline
$(6,14,1,21)$ & $\begin{pmatrix}
        \frac{3\sqrt{6}+\sqrt{14}}{2} & 3 \\
        3 &
        \frac{3\sqrt{6}-\sqrt{14}}{2}
    \end{pmatrix}$ & No & $(0,1,0)$\\
\hline
$(14,6,21,1)$ & $\begin{pmatrix}
        0 & 1 \\
        -1 &
        0
    \end{pmatrix}$ & Yes & $(1,1,0)$\\
\hline
$(2,42,3,7)$ & $\begin{pmatrix}
        0 & \frac{\sqrt{3}+\sqrt{7}}{2} \\
        \frac{\sqrt{3}-\sqrt{7}}{2} &
        0
    \end{pmatrix}$ & Yes & $(0,0,1)$\\
\hline
$(42,2,7,3)$ & $\begin{pmatrix}
        \sqrt{2} & \frac{3\sqrt{7}+5\sqrt{3}}{2} \\
        \frac{3\sqrt{7}-5\sqrt{3}}{2} &
        -\sqrt{2}
    \end{pmatrix}$ & No & $(0,1,1)$\\
\hline
$(3,7,2,42)$ & $\begin{pmatrix}
        \sqrt{3} & \sqrt{2} \\
        \sqrt{2} &
        \sqrt{3}
    \end{pmatrix}$ & No & $(0,1,1)$\\
\hline
$(7,3,42,2)$ & $\begin{pmatrix}
        \frac{\sqrt{7}+\sqrt{3}}{2} & 0 \\
        0 &
        \frac{\sqrt{7}-\sqrt{3}}{2}
    \end{pmatrix}$ & Yes & $(1,1,1)$\\
 
\hline
\caption{Elements of the Group $\Gamma^{\prime\prime}/H$}
\end{longtable}

\subsection{Commensurability Class of $\Gamma$}
\label{subsection: type of H as arith}
Recall that arithmetic subgroups of $\SL(2,\RR)$ up to commensurability and conjugacy have been completely classified, see for example \cite[Chapter 6]{morris2015introduction}. For reader's convenience, we gather these facts as below

\begin{defn}
    A closed subgroup $\Delta < \SL(2,\RR)$ is cocompact if $\SL(2,\RR)/\Delta$ is compact.  
\end{defn}

We borrow the following proposition \cite[Chapter 6, Proposition 6.1.5]{morris2015introduction}
\begin{prop}
\label{propositon:only noncocompact}
    $\SL(2,\ZZ)$ is the only noncocompact, arithmetic subgroup of $\SL(2,\RR)$ up to commensurability and conjugacy. 
\end{prop}

What is left is those cocompact arithmetic subgroups. We first recall what cocompact arithmetic subgroups can appear according to \cite[Proposition 6.2.4, Propsition 6.2.5 and Proposition 6.2.6]{morris2015introduction}.
\begin{defn}
\begin{enumerate}
    \item For any field $F$, and any nonzero $a,b \in F$, the corresponding quaternion algebra over $F$ is the ring
    \[
    \HH_F^{a,b} = \{p+qi+rj+sk|p,q,r,s\in F\},
    \]
    where
    \begin{itemize}
        \item addition is defined in the obvious way, and
        \item multiplication is determined by the relations
        \[
        i^2 = a, j^2 = b, ij=k=-ji
        \]
        together with the requirement that every element of $F$ is in the center of $\HH_F^{a,b}$. 
    \end{itemize}
    \item The reduced norm of $x = p+qi+rj+sk \in \HH_F^{a,b}$ is
    \[
    \mathrm{N}_{red}(x) = x\overline{x} = p^2-aq^2-br^2+abs^2 \in F
    \]
    where $\overline{x} = p-qj-rj-sk$ is the conjugate of $x$. 
\end{enumerate}
\end{defn}

For $a,b \in \RR^+$, there is an isomorphism of rings $\psi\colon \HH^{a,b}_\RR \xrightarrow{\sim} \mathrm{M}_{2\times2}(\RR)$ sending $i$ to $\begin{pmatrix}
    \sqrt{a}&0\\0&-\sqrt{a}
\end{pmatrix}$, $j$ to $\begin{pmatrix}
    0&1\\b&0
\end{pmatrix}$ and $k$ to $\begin{pmatrix}
    0&\sqrt{a}\\-b\sqrt{a} & 0
\end{pmatrix}$ such that $\mathrm{N}_{red}(x) = \det(\psi(x))$. Let $\SL(1, \HH^{a,b}_F)\coloneqq \{g\in \HH_F^{a,b}|\mathrm{N}_{red}(x) = 1\}$, then under the identification above, we have $\SL(2,\RR) = \SL(1, \HH^{a,b}_\RR)$. We recall the following fact according to \cite[Proposition 6.2.4, Proposition 6.2.5 and Proposition 6.2.6]{morris2015introduction}
\begin{prop}
\label{proposition:arithsubgrp}
    Every cocompact and arithmetic subgroup of $\SL(2,\RR)$ up to commensurability and conjugacy is of the following form
    \begin{itemize}
        \item $\SL(1, \HH_\ZZ^{a,b})< \SL(1, \HH_\RR^{a,b})\cong\SL(2,\RR)$, where $a,b \in \ZZ^+$ and $(0,0,0,0)$ is the only integer solution $(p,q,r,s)$ of the Diophantine equation\[
        w^2-ax^2-by^2+abz^2 = 0
        \]
        \item $\SL(1, \HH_\mathcal{O}^{a,b})<\SL(1,\HH_\RR^{a,b}) \cong \SL(2,\RR)$, where $\mathcal{O}$ is the ring of integers of a totally real algebraic number field $F \ne \QQ$ and $a,b\in \mathcal{O}$ are positive such that $\sigma(a)$ and $\sigma(b)$ are negative for every place $\sigma \ne \Id$. 
    \end{itemize}
\end{prop}

Next we show $\Gamma$ is a cocompact arithmetic subgroup of $\SL(2,\RR)$.

\begin{prop}
\label{proposition:cocompactness}
    The curve $\mathbb{D}_T/\Gamma$ is compact, and therefore not isogenous to $\PSL(2,\ZZ)\bs\HH$. 
\end{prop}
\begin{proof}
    By Proposition \ref{propositon:only noncocompact}, we only need to show that $\DD_T/\Gamma$ is compact. This can be shown by considering the Baily-Borel compactification. In fact, the boundaries of Baily-Borel compactification are the isotropic subspaces of $T\otimes \QQ$ \cite{baily1966compactification}. However, there is no isotropic vector. The claim is equivalent to that the equation $-x^{2}+xy+5y^{2}=14z^{2}$ has no coprime integer solution.
    
    If there is an integer solution $x,y,z$ with $\gcd(x,y,z)=1$, then  $x,y$ must be even integers. So $4$ divides $14z^{2}$ and thus $z$ is forced to be an even integer, a contradiction to the coprime assumption.
\end{proof}

\begin{rmk}
\label{rmk:compactnessfromGIT}
    The compactness of $\Gamma\backslash\mathbb{D}_{T}$ can also be deduced from the GIT analysis in Proposition \ref{singuL27}.

    More precisely, $\overline{\mathcal{F}}_{L_{2}(7)}\cong \mathbb{P}^{1}$ can be obtained from $\mathcal{F}_{L_{2}(7)}^{\mathrm{ADE}}$ by adding one point that corresponds to the determinantal cubic fourfold. And note that the point of determinantal cubic fourfold is of codimension 1. By Theorem \ref{thm:Yu--Zhengperiodmapsymmcubic4fold}, the period map:
    \begin{align*}
        \mathcal{F}_{L_{2}(7)}^{\mathrm{ADE}}\rightarrow \Gamma_{T}\bs(\mathbb{D}- \calH_{*})
    \end{align*}
extends to 
    \begin{align}
    \label{equation:periodmapL27}
\overline{\mathcal{F}}_{L_{2}(7)}\cong \overline{\Gamma_{T}\bs\mathbb{D}}^{\mathcal{H}_*}=\overline{\Gamma_{T}\bs\mathbb{D}}^{BB}.
    \end{align}
From analysis of singular cubic fourfolds in $\overline{\mathcal{F}}_{L_{2}(7)}$, the image of map \ref{equation:periodmapL27} should be $\Gamma_{T}\bs\mathbb{D}$. Therefore, $\Gamma_{T}\bs\mathbb{D}$ is itself compact.
\end{rmk}

In fact, we can further determine the type of $\HH$ as an arithmetic subgroup of $\SL(2, \RR)$

\begin{thm}
\label{theorem:typeofdomainL27}
   The curve $\DD_{T}/\Gamma$ is isogenous to $\HH/\SL(1, \HH_\ZZ^{21,6})$. 
\end{thm}

\begin{proof}
    We show $\SL(1, \HH_\ZZ^{21,6})<H$ up to a conjugation.

    Consider $A=\begin{pmatrix}
        0 & 6^{-\frac{1}{4}}\\
        -6^{\frac{1}{4}} & 0
    \end{pmatrix}$, 
     then 
     \[
     AHA^{-1}=\{
    \begin{pmatrix}
        \frac{u+v\sqrt{21}}{2} & \frac{3w-x\sqrt{21}}{6} \\
        3w+x\sqrt{21} &
        \frac{u-v\sqrt{21}}{2}
    \end{pmatrix}|u,v,w,x \in \ZZ, u^{2}-21v^{2}-6w^{2}+14x^{2}=4\}.
    \]

    For equation $u^{2}-21v^{2}-6w^{2}+14x^{2}=4$, we analyze modulo $3$ and obtain $u^{2}+2x^{2}\equiv 1  \, (\mathrm{mod}\,3)$. We must have $3| x$, otherwise $u^{2}\equiv -1 \, (\mathrm{mod}\,3)$, which is impossible. Replace $x$ by $3x$, and one obtains 
    \[
    AHA^{-1}=\{
    \begin{pmatrix}
        \frac{u+v\sqrt{21}}{2} & \frac{w-x\sqrt{21}}{2} \\
        6(\frac{w+x\sqrt{21}}{2}) &
        \frac{u-v\sqrt{21}}{2}
    \end{pmatrix}|u,v,w,x \in \ZZ, u^{2}-21v^{2}-6w^{2}+126x^{2}=4\}.
    \]

    And $\SL(1, \HH_\ZZ^{21,6})$ is a subgroup of $AHA^{-1}$ consisting of matrices where $u,v,w,x$ are even integers.

    Therefore, $\SL(1, \HH_\ZZ^{21,6})<\Gamma^{\prime \prime}$ and thus they are commensurable with each other, since both are arithmetic subgroups of $\SL(2,\RR)$. 

    Modulo 7 and using infinite descent, one can show that equation $u^{2}-21v^{2}-6w^{2}+126x^{2}=0$ does not have non-trivial integer solution. Therefore, $\SL(1, \HH_\ZZ^{21,6})$ satisfies the conditions in Proposition \ref{proposition:arithsubgrp}. 
\end{proof}


\section{Periods of the $F_{21}$-family}
\label{section: period of F21}
For convenience, we use notation $T$ instead of $T_2$, $\Gamma$ instead of $\Gamma_{T_2}$ and $\Phi$ instead of $\Phi_{2}$ inside this section.

\subsection{Calculation of the Lattice $T$}
\label{conslatF21}
The lattice $T$ for smooth cubic fourfolds with $F_{21}$-action was calculated by \cite[Table 1]{marquand2025cubic}. We review the calculation for reader’s convenience.

Now let $G=F_{21}$. Let $S$ be the coinvariant lattice for the action of $F_{21}$ on $\Lambda_0$. By \cite[Lemma 4.2]{laza2022automorphisms}, $(G,S)$ is a Leech pair satisfying the condition in Lemma \ref{lemma: leech pair}. Then by Lemma \ref{lemma: leech pair}, there is a primitive embedding $S\hookrightarrow \LL$. Thus the $F_{21}$-action on $S$ can be extended to $\mathbb{L}$, such that $S$ is the coinvariant lattice in $\mathbb{L}$. It is known from \cite{laza2022automorphisms} that the invariant sublattice $\mathbb{L}^{G}$ corresponds to \cite[Table 1, item 52]{hohn2016290}, hence $\mathbb{L}^{G}$ has discriminant form $7^{+3}$. 


Similar as \S\ref{conslatL27}, $T$ has discriminant form $q_{T}=7^{+3}3^{-1}$ and such $T$ is unique by \cite[Theorem 1.13.2]{Nikulin_1980}. By direct construction, we know there exists an integral basis for $T$, such that the Gram matrix is:
\begin{equation}
\label{equation: Gram matrix for F21 case}
\begin{pmatrix}
    -2 & 1 & 0 & 0\\
    1 & 10 & 0 & 0\\
    0 & 0 & 0 & 7\\
    0 & 0 & 7 & 0
\end{pmatrix}
\end{equation}

\subsection{Characterization of $\Gamma$}
\label{subsection: char of Gamma F21}
By abuse of notation, we still denote the above matrix by $T$. Also from \cite[\S 4.1.2]{laza2022automorphisms}, we know that $(F_{21}, S)$ is a saturated Leech pair and therefore $F_{21} = \{h \in \widehat{\Gamma}|h|_T = \Id_T\}$. Also let $N$ be the normalizer of $F_{21}$ inside $\widehat{\Gamma}$. The same as the discussion in the last section, we have a well-defined restriction map $N \to \rO^+(T)$. Therefore, we may just assume $\Gamma = \mathrm{Im}(N \to \rO^+(T))$. 

Recall that from the above calculation, the discriminant group $A_{\Lambda_0}, A_S$ and $A_T$ are isomorphic to $\ZZ/3\ZZ, (\ZZ/7\ZZ)^3$ and $(\ZZ/7\ZZ)^3 \oplus \ZZ/3\ZZ$ respectively. Also from \cite[\S 4]{hohn2016290}, we know that $\rO(S) \to A_S$ is surjective. Therefore, we can completely copy the proof of Proposition \ref{gam} and obtain the following
\begin{prop}
    An element $f \in \rO^+(T)$ is in $\Gamma$ if and only if $\overline{f}|_{\ZZ/3\ZZ}$ is trivial. 
\end{prop}
Up to a multiplication by $-\Id$, we may just consider $\rO^+(T)$ itself. More precisely
\begin{prop}
    As a group, $\rO^+(T)$ is generated by $\Gamma$ and $-\Id$. $\Gamma$ is an index $2$ subgroup of $\rO^+(T)$. 
\end{prop}
An important thing to point out is that unlike the $L_2(7)$ case, $-\Id$ in $\rO^+(T)$ has determinant $1$. Therefore, we cannot find the copy of $\Gamma$ to be $\rO^+(T)\cap \SO(2,2)$. Instead, we first find an element \[P = \begin{pmatrix}
    1&0&0&0\\0&1&0&0\\0&0&0&-1\\0&0&-1&0
\end{pmatrix}\] 
$\in \rO^+(T)$ which is of determinant $-1$, and thus form the following result.
\begin{prop}
The group $\SO^+(T)$ is a normal subgroup of $\rO^+(T)$ with index two, and the quotient group $\rO^+(T)/\SO^+(T)$ is generated by $P$ above. 
\end{prop}

We will first compute $\SO^+(T)$ and give a modular interpretation of $\mathbb{D}_T/\SO^+(T)$. The latter will be a $2$ to $1$ cover of the $F_{21}$-family. 

\subsection{Tube Domain Model for the $F_{21}$-family}
\label{subsection: tube domain model of the F21}
Now let $e_1,e_2,e_3,e_4 \in T$ be the integral basis for which we have the above Gram matrix \eqref{equation: Gram matrix for F21 case}. Then we have $(e_1)^2=-2,(e_2)^2=10,(e_1,e_2)=1,(e_3,e_4) = 7$. Now take $f_1,f_2,g_1,g_2 \in T\otimes \RR$ such that if we denote 

\begin{align}
\label{equation:transitionmatrix}
   B = \begin{pmatrix}
    -\frac{\sqrt{21}}{42}-\frac{1}{2}&0&0&-\frac{\sqrt{21}}{42}+\frac{1}{2}\\
    -\frac{\sqrt{21}}{21} & 0 & 0 & -\frac{\sqrt{21}}{21}\\
    0&1&0&0\\
    0&0&-\frac{1}{7}&0
\end{pmatrix}, 
\end{align}

we have $\begin{pmatrix}
    f_1&f_2&g_2&g_1
\end{pmatrix} = \begin{pmatrix}
    e_1&e_2&e_3&e_4
\end{pmatrix}B$

Therefore, the Gram matrix with respect to $f_1,f_2,g_2,g_1$ is just \[
B^TTB=
\begin{pmatrix}
    0&0&0&1\\0&0&-1&0\\0&-1&0&0\\1&0&0&0
\end{pmatrix}.\]

With respect to this basis, there is an accidental isometry from $T\otimes\CC$ to $\mathrm{M}_2(\CC)$, sending $z_{11}f_1+z_{12}f_2+z_{21}g_2+z_{22}g_1$ to $\begin{pmatrix}
    z_{11} & z_{12}\\z_{21} & z_{22}
\end{pmatrix}$. Here $\mathrm{M}_2(\CC)$ is endowed with the quadratic form $2(z_{11}z_{22}-z_{12}z_{21})$. 

Now consider the group $\SL(2,\RR)\times \SL(2,\RR)$. Each element $(h_1,h_2) \in \SL(2,\RR)\times \SL(2,\RR)$ can be seen as an isometry of the space $\mathrm{M}_2(\CC)$ via the action $(h_1,h_2)\cdot M = h_1Mh_2^T$. In this way, we obtain a surjective morphism $\SL(2,\RR)\times \SL(2,\RR) \to \SO^+(2,2;\RR)$ with kernel $\pm(\I,\I)$. When pass to $\SL(2,\RR)\times \SL(2,\RR) \to \mathrm{PSO}^+(2,2;\RR)$ we have kernel being $(\pm\I,\pm\I)$. More explicitely, set $h = (h_1,h_2) = (\begin{pmatrix}
    a_1&b_1\\c_1&d_1
\end{pmatrix},\begin{pmatrix}
    a_2&b_2\\c_2&d_2
\end{pmatrix})$ and let $A(h)$ be the matrix of the action of $h$ with respect to the basis $f_1,f_2,g_2,g_1$, then we have 
\[
A(h) = \begin{pmatrix}
    a_1h_2&b_1h_2\\
    c_1h_2&d_1h_2
\end{pmatrix} = h_1\otimes h_2.
\] 
Now denote by $\Phi \colon \SL(2,\RR)\times \SL(2,\RR) \to \SO^+(2,2)$ the morphism sending $h$ to $BA(h)B^{-1}$, the matrix of the action of $h$ with respect to the basis $e_1, e_2, e_3, e_4$. 

Now let $\HH^2 = \{(z_1,z_2)\in \CC^2|\Im(z_1)>0,\Im(z_2)>0\}$ be the product of upper half planes. The group $\SL(2,\RR)\times \SL(2,\RR)$ acts on $\HH^2$ via $h\cdot (z_1,z_2) = (\frac{a_1z_1+b_1}{c_1z_1+d_1}, \frac{a_2z_2+b_2}{c_2z_2+d_2})$. There is a classical tube domain model identification of $\mathbb{D}_T$ as $\HH^2$. Precisely, there is a biholomorphic morphism $\mathbb{H}^2 \to \mathbb{D}_T$, sending $(z_1,z_2)$ to $z_1z_2f_1+z_1f_2+z_2g_2+g_1$. The latter is $\begin{pmatrix}
    z_1z_2 & z_1\\z_2&1
\end{pmatrix} = \begin{pmatrix}
    z_1\\1
\end{pmatrix}\begin{pmatrix}
    z_2&1
\end{pmatrix}$ when identified as a matrix. Therefore, we obtain an equivariant diagram:
\[ 
\begin{tikzcd}
\SL(2, \mathbb{R}) \times \SL(2, \mathbb{R}) \arrow{r}{\Phi} \arrow[symbol=\circlearrowright]{d} & \SO^+(2,2) \arrow[symbol=\circlearrowright]{d} \arrow[hook]{r} & \mathrm{O}(2,2) \arrow[symbol=\circlearrowright]{d} \\
\mathbb{H}^2 \arrow{r}{\cong} & \mathbb{D}_T \arrow[hook]{r} & \mathbb{P}(T \otimes \CC)
\end{tikzcd}
\]

First we find the inverse image $\Phi^{-1}(\SO^+(T))$. Note that $h \in \Phi^{-1}(\SO^+(T))$ if and only if $\Phi(h)$ has integer coefficients. 

\begin{prop}
\label{strulem}
    Denote $\Gamma^\prime = \Phi^{-1}(\SO^+(T))$, then we have $h \in \Gamma^\prime$ if and only if
    \[
    h =(\begin{pmatrix}
        \frac{\alpha_1\sqrt{\alpha}+\beta_1\sqrt{\beta}}{2}
        &\frac{1}{\sqrt{7}}\frac{\gamma_1\sqrt{\gamma}+\delta_1\sqrt{\delta}}{2}\\
        \sqrt{7}\frac{\gamma_2\sqrt{\gamma}-\delta_2\sqrt{\delta}}{2}&\frac{\alpha_2\sqrt{\alpha}-\beta_2\sqrt{\beta}}{2}
    \end{pmatrix},
    \begin{pmatrix}
        \frac{\alpha_2\sqrt{\alpha}+\beta_2\sqrt{\beta}}{2}
        &\sqrt{7}\frac{\gamma_2\sqrt{\gamma}+\delta_2\sqrt{\delta}}{2}\\
        \frac{1}{\sqrt{7}}\frac{\gamma_1\sqrt{\gamma}-\delta_1\sqrt{\delta}}{2}&\frac{\alpha_1\sqrt{\alpha}-\beta_1\sqrt{\beta}}{2}
    \end{pmatrix})
    \]
    where 
\begin{itemize}
    \item $(\alpha,\beta,\gamma,\delta) = (1,21,3,7)$ up to permutation by $\mathrm{V}=\langle(13)(24),(12)(34)\rangle < \mathrm{S}_4$,  
    \item $\alpha_i,\beta_i,\gamma_i,\delta_i \in \ZZ$ such that $\alpha_i-\beta_i,\gamma_i-\delta_i\in 2\ZZ$,
    \item $\alpha_1\alpha_2\alpha - \beta_1\beta_2\beta-\gamma_1\gamma_2\gamma+\delta_1\delta_2\delta = 4$ and $\alpha_1\beta_2-\alpha_2\beta_1=\gamma_1\delta_2-\gamma_2\delta_1$. 
\end{itemize}
    (The last condition is equivalent to saying that the determinants of $h_1$ and $h_2$ are $1$.)
\end{prop}

\begin{proof}
    Set $h = (h_1,h_2) = (\begin{pmatrix}
    a_1&b_1\\c_1&d_1
\end{pmatrix},\begin{pmatrix}
    a_2&b_2\\c_2&d_2
\end{pmatrix})$ and suppose 
\[\Phi(h) = \begin{pmatrix}
    n_1&n_2&n_3&n_4\\
    n_5&n_6&n_7&n_8\\
    n_9&n_{10}&n_{11}&n_{12}\\
    n_{13}&n_{14}&n_{15}&n_{16}
\end{pmatrix}
=
BA(h)B^{-1}
\]
where $n_i \in \ZZ$ and the expression of $B$ is given in \ref{equation:transitionmatrix}. Writing it out, we see that
\begin{equation*}
A(h)=  \begin{pmatrix}
    a_1h_2&b_1h_2\\
    c_1h_2&d_1h_2
\end{pmatrix}
= \begin{pmatrix}
    a_1 a_2&a_1 b_2 & b_1 a_2 & b_1 b_2 \\
    a_1 c_2&a_1 d_2 & b_1 c_2 & b_1 d_2 \\
    c_1 a_2 & c_1 b_2 & d_1 a_2 & d_1 b_2 \\
    c_1 c_2 & c_1 d_2 & d_1 c_2 & d_1 d_2
\end{pmatrix}
\end{equation*}
equals to 
\begin{frame}
\footnotesize
\setlength{\arraycolsep}{2.5pt} 
\medmuskip = 1mu 
\begin{align*}
\left( \begin{array}{llllll} 
\frac{21(n_1+n_6)+(n_1+n_2+5n_5-n_6)\sqrt{21}}{42} & \frac{n_7-2n_3-n_7\sqrt{21}}{2} & \frac{2n_4-n_8+n_8\sqrt{21}}{14} & \frac{21(-n_1+n_5+n_6)+(n_1+2n_2-11n_5-n_6)\sqrt{21}}{42}  \\
\frac{-21n_9-(n_9+2n_{10})\sqrt{21}}{42} & n_{11} & -\frac{n_{12}}{7} & \frac{21n_9-(n_9+2n_{10})\sqrt{21}}{42} \\
\frac{21n_{13}+(n_{13}+2n_{14})\sqrt{21}}{6} & -7n_{15} & n_{16} & \frac{-21n_{13}+(n_{13}+2n_{14})\sqrt{21}}{6}  \\
 \frac{21(-n_1+n_5+n_6)-(n_1+2n_2-11n_5-n_6)\sqrt{21}}{42} & \frac{-n_7+2n_3-n_7\sqrt{21}}{2} & \frac{-2n_4+n_8+n_8\sqrt{21}}{14} & \frac{21(n_1+n_6)-(n_1+n_2+5n_5-n_6)\sqrt{21}}{42}    \\
\end{array} \right)
\end{align*}
\end{frame}

Now if we set $u_1 = a_1+d_2,v_1=a_1-d_2,w_1=\sqrt{7}(b_1+c_2),x_1=\sqrt{7}(b_1-c_2)$ and $u_2 = a_2 +d_1,v_2=a_2-d_1,w_2=\frac{1}{\sqrt{7}}(b_2+c_1),x_2=\frac{1}{\sqrt{7}}(b_2-c_1)$, then using information from the above matrix, we can prove various properties of these numbers. 

First, we claim that $u_i^2,v_i^2,w_i^2,x_i^2\in\ZZ$. For example, $u_1^2 = a_1^2+2a_1d_2+d_2^2$, the term $a_1 d_2 = n_{11}$ is an integer. The terms $a_1^2=\det(a_1 h_2)$ and $d_2^2=\det(d_2 h_1)$ can be expressed in terms of $n_i$ and we can obtain 
\begin{equation*}
u_1^2=n_{1}n_{11}-n_{3}n_{9}+n_{6}n_{11}-n_{7}n_{10}+2n_{11}.
\end{equation*}
Therefore, $u_1^2 \in \ZZ$. The others are similar. 

Second, note that $u_1u_2 = a_1a_2+d_1d_2+a_1d_1+a_2d_2, v_1v_2 = a_1a_2+d_1d_2-(a_1d_1+a_2d_2),w_1w_2 = b_1b_2+c_1c_2+b_1c_1+b_2c_2, x_1x_2=b_1b_2+c_1c_2-(b_1c_1+b_2c_2)$. Among them, $a_1a_2+d_1d_2,b_1b_2+c_1c_2$ are easily seen to be integers. Taking the $(2,3)$ rows and $(2,3)$ columns, the corresponding submatrix has determinant $a_1d_1a_2d_2-b_1c_1b_2c_2  = (a_1d_1-b_1c_1)a_2d_2+b_1c_1(a_2d_2-b_2c_2) = a_2d_2+b_1c_1 = a_2d_2+a_1d_1-1=b_1c_1+b_2c_2+1$, which is an integer. Therefore, $u_1u_2,v_1v_2,w_1w_2,x_1x_2 \in \ZZ$. 

Combining the two observations above, we can set $u_i= \alpha_i\sqrt{\alpha}$, $v_i = \beta_i\sqrt{\beta}$, $w_i = \gamma_i\sqrt{\gamma}$, $x_i = \delta_i\sqrt{\delta}$, where $\alpha_i,\beta_i,\gamma_i,\delta_i$ are integers and $\alpha,\beta,\gamma,\delta$ are positive square-free integers. Next we want to determine the possibilities of $\alpha$, $\beta$, $\gamma$, $\delta$. Using the same tricks before, we obtain $\sqrt{21}u_2v_2, \sqrt{21}w_2x_2, \sqrt{3}u_2w_2 \in \ZZ$. Therefore, we can set $\alpha\beta=21y_1^2, \gamma\delta=21y_2^2, \alpha\gamma = 3y_3^2$ with $y_1,y_2,y_3\in \ZZ^+$, square free, $y_1,y_2$ not divided by $3,7$ and $y_3$ not divided by $3$. Note that $\alpha\gamma = 3y_3^2$ gives two possibilities. One is $\alpha = 3\gamma$ and the other is $\gamma = 3\alpha$. Suppose $\alpha = 3\gamma$, then $\gamma$ is not divided by $3$. Therefore $\gamma = 7y_2$ or $\gamma = y_2$. Meanwhile, since $\alpha$ is divided by $3$, we obtain $\alpha = 21y_1$ or $3y_1$. Suppose $\gamma = 7y_2$, then $\alpha \ne 3y_1$, otherwise $\alpha = 3\gamma = 21y_2$ will imply that $y_1$ is divided by $7$. Similarily, when $\gamma = y_2$, we do not have $\alpha = 21y_1$. The same argument can run again for $\gamma = 3\alpha$. In summary, we obtain that $(\alpha,\beta,\gamma,\delta) = (y_2,21y_2,3y_2,7y_2)$ up to permutation by $\mathrm{V} < \mathrm{S}_4$. 

Determinant of $h_1$ or $h_2$ being $1$ tells you that $\alpha_1\alpha_2\alpha-\beta_1\beta_2\beta-\gamma_1\gamma_2\gamma+\delta_1\delta_2\delta = 4$ and $\alpha_1\beta_2-\alpha_2\beta_1=\gamma_1\delta_2-\gamma_2\delta_1$. From the first equation above, we note that $4$ is divided by $y_2$ and therefore $y_2 = 1$ or $2$. 

Now we already know $h$ must be of the form in the statement of this lemma, but only assuming $(\alpha,\beta,\gamma,\delta) = (1,21,3,7)$ or $(2,42,6,14)$ up to permutation by $\mathrm{V} < \mathrm{S}_4$ and having determinants $1$. By writing out the matrix $\Phi(h)=BA(h)B^{-1}$, an elementary case-by-case observation shows that $\Phi(h)$ has integer coefficients if and only if $\frac{\alpha_i^2-\beta_i^2}{4},\frac{\gamma_i^2-\delta_i^2}{4}\in \ZZ$ when $(\alpha,\beta,\gamma,\delta) = (1,21,3,7)$ up to permutation and $\frac{\alpha_i^2-\beta_i^2}{2},\frac{\gamma_i^2-\delta_i^2}{2}\in \ZZ$ when $(\alpha,\beta,\gamma,\delta) = (2,42,6,14)$ up to permutation. They are all the same as $\alpha_i-\beta_i,\gamma_i-\delta_i\in 2\ZZ$. 

Finally we show $(\alpha,\beta,\gamma,\delta) = (2,42,6,14)$ (up to permutation) is impossible. For example, when $(\alpha,\beta,\gamma,\delta) = (2,42,6,14)$, we would have $\alpha_1\alpha_2-21\beta_1\beta_2-3\gamma_1\gamma_2+7\delta_1\delta_2=2$. Mod out $4$ and using $\alpha_1\beta_2-\alpha_2\beta_1=\gamma_1\delta_2-\gamma_2\delta_1$, this gives us $(\alpha_1-\beta_1)(\alpha_2-\beta_2)+(\gamma_1-\delta_1)(\gamma_2-\delta_2) \equiv 2 \, (\mathrm{mod}\,4)$ but not $0$. Therefore, the whole lemma has been proven. 
\end{proof}

Let us denote the subgroup of elements of type $(1,21,3,7)$ as $H$. Similar as Proposition \ref{normal}, we obtain the following result.
\begin{prop}
    $H$ is an index $4$ normal subgroup of $\Gamma^{\prime}$ above. With $\Gamma^{\prime}/H \cong \ZZ/2\ZZ\oplus\ZZ/2\ZZ$. 
\end{prop}

We also make a table to make its structure clear

\begin{longtable}{ |p{2cm}||p{6cm}|  }

 \hline
 Type& Representative \\
 \hline
 $(1,21,3,7)$ & $(\Id,\Id)$  \\
 \hline
 $(21,1,7,3)$ & $(\Id,-\Id)$  \\
    \hline
 $(3,7,1,21)$ & $(\begin{pmatrix}
        0 & \frac{1}{\sqrt{7}} \\
        -\sqrt{7} &
        0
    \end{pmatrix},\begin{pmatrix}
        0 & -\sqrt{7} \\
        \frac{1}{\sqrt{7}} &
        0
    \end{pmatrix})$  \\
    \hline
 $(7,3,21,1)$ & $(\begin{pmatrix}
        0 & \frac{1}{\sqrt{7}} \\
        -\sqrt{7} &
        0
    \end{pmatrix},\begin{pmatrix}
        0 & \sqrt{7} \\
        -\frac{1}{\sqrt{7}} &
        0
    \end{pmatrix})$  \\
    \hline
    \caption{Elements of $\Gamma^{\prime}/H$}
\end{longtable}

Note that $(\Id,-\Id)$ has no effect on $\mathbb{H}^2$. Therefore we can view $\mathbb{D}_T/\Gamma$ as $\HH^2$ quotient by $H$ and an involution
\[
\widehat{F} = (\begin{pmatrix}
        0 & \frac{1}{\sqrt{7}} \\
        -\sqrt{7} &
        0
    \end{pmatrix},\begin{pmatrix}
        0 & \sqrt{7} \\
        -\frac{1}{\sqrt{7}} &
        0
    \end{pmatrix}). 
\]

Moreover, $H$ is isomorphic to a Hilbert modular group. To make it precise, let $F = \QQ(\sqrt{21})$ and $\mathcal{O}_F = \ZZ[\frac{1+\sqrt{21}}{2}]$ be its ring of integers. It is well known that $\mathcal{O}_F$ is a PID. Let $\mathfrak{p} = (\frac{7+\sqrt{21}}{2})$ be its only prime ideal lying over $7 \in \ZZ$. The Hilbert modular group corresponding to $\mathfrak{p}$ is
\[
\Gamma(\mathcal{O}_F\oplus\mathfrak{p}) = \{\begin{pmatrix}
    a&b\\c&d
\end{pmatrix}\in \SL(2,F)|a,d\in \mathcal{O}_F,b\in \mathfrak{p}^{-1},c\in\mathfrak{p}\}. 
\]
It is a subgroup of $\SL(2,F)$. 

Now let $\varphi\colon \SL(2,F) \to \SL(2,\RR)\times\SL(2,\RR)$ be a group homomorphism, sending $\begin{pmatrix}
    a&b\\c&d
\end{pmatrix}$ to $(\begin{pmatrix}
    a&b\\c&d
\end{pmatrix},\begin{pmatrix}
    a^\prime&b^\prime\\c^\prime&d^\prime
\end{pmatrix})$. Here we use $a^\prime$ to denote the Galois conjugate of $a$ inside $F$. Meanwhile, let $C = \begin{pmatrix}
    0&-1\\1&0
\end{pmatrix}$ and $\psi_C$ be an automorphism of $\SL(2,\RR)\times\SL(2,\RR)$ sending $((h_1,h_2)$ to $(h_1,Ch_2C^{-1})$. We will have
\begin{prop}
    The morphism $\psi \circ \varphi\colon\SL(2,F)\to \SL(2,\RR)\times\SL(2,\RR)$ maps $\Gamma(\mathcal{O}_F\oplus\mathfrak{p})$ isomorphically onto $H$. 
\end{prop}
\begin{proof}
    According to Lemma \ref{strulem}, we know that $h \in H$ if and only if 
    \begin{align*}
        h &=(\begin{pmatrix}
        \frac{\alpha_1+\beta_1\sqrt{21}}{2}
        &\frac{1}{7}\frac{\gamma_1\sqrt{21}+7\delta_1}{2}\\
        \frac{\gamma_2\sqrt{21}-7\delta_2}{2}&\frac{\alpha_2-\beta_2\sqrt{21}}{2}
    \end{pmatrix},
    \begin{pmatrix}
        \frac{\alpha_2+\beta_2\sqrt{21}}{2}
        &\frac{\gamma_2\sqrt{21}+7\delta_2}{2}\\
        \frac{1}{7}\frac{\gamma_1\sqrt{21}-7\delta_1}{2}&\frac{\alpha_1-\beta_1\sqrt{21}}{2}
    \end{pmatrix})\\
    &= (\begin{pmatrix}
        \frac{\alpha_1+\beta_1\sqrt{21}}{2}
        &\frac{1}{7}\frac{\gamma_1\sqrt{21}+7\delta_1}{2}\\
        \frac{\gamma_2\sqrt{21}-7\delta_2}{2}&\frac{\alpha_2-\beta_2\sqrt{21}}{2}
    \end{pmatrix},
    C\begin{pmatrix}
        \frac{\alpha_1-\beta_1\sqrt{21}}{2}
        &\frac{1}{7}\frac{-\gamma_1\sqrt{21}+7\delta_1}{2}\\
        \frac{-\gamma_2\sqrt{21}-7\delta_2}{2}&\frac{\alpha_2+\beta_2\sqrt{21}}{2}
    \end{pmatrix}C^{-1}
    )\\
    &= \psi \circ \varphi(\begin{pmatrix}
        \frac{\alpha_1+\beta_1\sqrt{21}}{2}
        &\frac{1}{7}\frac{\gamma_1\sqrt{21}+7\delta_1}{2}\\
        \frac{\gamma_2\sqrt{21}-7\delta_2}{2}&\frac{\alpha_2-\beta_2\sqrt{21}}{2}
    \end{pmatrix})
    \end{align*}
    such that $\alpha_i-\beta_i,\gamma_i-\delta_i \in 2\ZZ$ and satisfies the "determinant $=1$" condition. 

    Therefore, $H$ lies in the image of $\psi \circ \varphi$. Meanwhile, $\psi\circ\varphi$ is clearly injective, thus we only need to prove that the inverse image of $H$ is just $\Gamma(\mathcal{O}_F\oplus\mathfrak{p})$. 

    Note that $\frac{\alpha_1+\beta_1\sqrt{21}}{2} = \frac{\alpha_1-\beta_1}{2}+\beta_1\frac{1+\sqrt{21}}{2} \in \mathcal{O}_F$, the same apllies to $\frac{\alpha_2-\beta_2\sqrt{21}}{2}$. Meanwhile, $\frac{-7\delta_2+\gamma_2\sqrt{21}}{2} = (-(\gamma_2+2\delta_2)+\frac{\gamma_2+\delta_2}{2}\frac{1+\sqrt{21}}{2})\frac{7+\sqrt{21}}{2} \in \mathfrak{p}$ and almost the same applies to $\frac{1}{7}\frac{\gamma_1\sqrt{21}+7\delta_1}{2}$. Therefore, the inverse image of $H$ lies in $\Gamma(\mathcal{O}_F\oplus\mathfrak{p})$. Conversely, each element in $\mathcal{O}_F$ can be written as $\frac{\alpha_1+\beta_1\sqrt{21}}{2}$ with $\alpha_1-\beta_1\in 2\ZZ$ and each element in $\mathfrak{p}$ can be written as $\frac{-7\delta_2+\gamma_2\sqrt{21}}{2}$ with $\delta_2-\gamma_2 \in 2\ZZ$. The same also applies to $\mathfrak{p}^{-1}$. Therefore, the proposition is proven. 
\end{proof}

Recall that in \cite[Chapter 2, \S7]{van2012hilbert}, we can call $\Gamma(\mathcal{O}_F\oplus \mathfrak{p})\backslash\HH^2$ a Hilbert modular surface. As a summary of above description of $\Gamma$, we have the following theorem: 
\begin{thm}
    There are natural morphisms
    \[   
    \Gamma(\mathcal{O}_F\oplus\mathfrak{p})\backslash\HH^2 \to \Gamma^\prime\backslash\HH^2 \xrightarrow{\sim} \SO^+(T)\backslash\mathbb{D}_T \to \Gamma\backslash\mathbb{D}_T\cong\rO^+(T)\backslash\DD_T.
    \]
    Each of them is of degree $2$ except the isomorphisms. The first degree $2$ morphism is a quotient by 
    \[
    (\begin{pmatrix}
        0 & \frac{1}{\sqrt{7}} \\
        -\sqrt{7} &
        0
    \end{pmatrix},\begin{pmatrix}
        0 & \sqrt{7} \\
        -\frac{1}{\sqrt{7}} &
        0
    \end{pmatrix})
    \]
    and the second is a quotient by
    \[
    P = \begin{pmatrix}
    1&0&0&0\\0&1&0&0\\0&0&0&-1\\0&0&-1&0
\end{pmatrix}.
    \]
  Therefore, the period domain $\Gamma\backslash\DD_T$ (for cubic fourfolds with order-$7$ action) is commensurable to the Hilbert modular surface $\Gamma(\mathcal{O}_F\oplus\mathfrak{p})\backslash\HH^2$.
\end{thm}

\bibliography{reference}
\end{document}